\documentclass[10pt]{amsart}

\usepackage{amsfonts}
\usepackage[utf8]{inputenc}
\usepackage[T1]{fontenc}
\usepackage{indentfirst}
\usepackage[margin=1in]{geometry} 
\usepackage{lipsum}
\usepackage{amsmath,amscd,amsthm,amssymb,mathrsfs,stmaryrd,color,tikz-cd,graphicx,tikz,mathtools,subfigure}
\tikzcdset{arrow style=tikz,diagrams={>=stealth}}
\usepackage{enumitem}
\usepackage[bookmarks, bookmarksdepth=1, colorlinks=true, linkcolor=blue, citecolor=red, urlcolor=blue,pagebackref]{hyperref}
\usepackage{subfiles}

\usetikzlibrary{matrix,arrows,decorations.pathmorphing,cd,topaths,calc,graphs}
\usepackage[all,cmtip]{xy}

\usepackage{relsize}

\makeatletter
\newcommand{\subalign}[1]{%
  \vcenter{%
    \Let@ \restore@math@cr \default@tag
    \baselineskip\fontdimen10 \scriptfont\tw@
    \advance\baselineskip\fontdimen12 \scriptfont\tw@
    \lineskip\thr@@\fontdimen8 \scriptfont\thr@@
    \lineskiplimit\lineskip
    \ialign{\hfil$\m@th\scriptstyle##$&$\m@th\scriptstyle{}##$\crcr
      #1\crcr
    }%
  }
}
\makeatother

\newcommand{\BA}{\mathbb{A}}

\newcommand{\BC}{\mathbb{C}}

\newcommand{\BH}{\mathbb{H}}

\newcommand{\BQ}{\mathbb{Q}}
\newcommand{\BR}{\mathbb{R}}

\newcommand{\BZ}{\mathbb{Z}}

\newcommand{\cB}{\mathcal{B}}

\newcommand{\cD}{\mathcal{D}}

\newcommand{\bG}{\mathbf{G}}

\newcommand{\bH}{\mathbf{H}}
\newcommand{\cH}{\mathcal{H}}

\newcommand{\bJ}{\mathbf{J}}

\newcommand{\cL}{\mathcal{L}}

\newcommand{\cO}{\mathcal{O}}

\newcommand{\cP}{\mathcal{P}}

\newcommand{\sP}{\mathscr{P}}

\newcommand{\cS}{\mathcal{S}}

\newcommand{\cT}{\mathcal{T}}

\newcommand{\fa}{\mathfrak{a}}

\newcommand{\fk}{\mathfrak{k}}

\newcommand{\fn}{\mathfrak{n}}

\newcommand{\bv}{\mathbf{v}}

\renewcommand{\emptyset}{\varnothing}

\renewcommand{\tilde}[1]{\widetilde{#1}}

\DeclareMathOperator{\supp}{Supp}

\DeclareMathOperator{\Ad}{Ad}

\newcommand{\GL}{\operatorname{GL}}
\newcommand{\SL}{\operatorname{SL}}
\newcommand{\SO}{\operatorname{SO}}
\newcommand{\SU}{\operatorname{SU}}

\newcommand{\fsl}{\mathfrak{sl}}

\newtheorem{theorem}{Theorem}[section]

\newtheorem{proposition}[theorem]{Proposition}

\newtheorem{corollary}[theorem]{Corollary}

\newtheorem{lemma}[theorem]{Lemma}

\theoremstyle{definition}

\theoremstyle{remark}

\newtheorem*{rmk}{Remark}

\newcommand{\RNum}[1]{\uppercase\expandafter{\romannumeral #1\relax}}

\usepackage[new]{old-arrows}

\usepackage{amssymb}

\usepackage[utf8]{inputenc}

\title[periods of eigenfunctions on arithmetic hyperbolic 3-manifolds]{Bounds for the periods of eigenfunctions on arithmetic hyperbolic 3-manifolds over surfaces}
\author{Jiaqi Hou}
\address{Department of Mathematics, Louisiana State University, 
Baton Rouge, LA 70803, USA}
\email{jhou7@lsu.edu}
\date{}

\usepackage{hyperref}
\numberwithin{equation}{section}
\begin{document}

\begin{abstract}
    Let $\psi$ be a Hecke-Maass form on a compact congruence arithmetic hyperbolic 3-manifold $X$, and let $Y$ be a totally geodesic surface in $X$ that is not necessarily closed. We obtain a power saving result over the local bound for the period of $\psi$ along $Y$, by applying the method of arithmetic amplification developed by Iwaniec and Sarnak.
\end{abstract}
\maketitle

\section{Introduction}

Let $X$ be a compact Riemannian manifold of dimension $n$ and let $\Delta$ denote the Laplace-Beltrami operator on $X$. If $\psi$ is a Laplace-Beltrami eigenfunction on $X$ satisfying $\Delta\psi+\lambda^2\psi=0$ with $\lambda>0$, we are interested in the bound of the integral of $\psi$ over a submanifold $Y$ of dimension $d$. We normalize $\psi$ by $\|\psi\|_2=1$. Zelditch \cite{zelditch1992kuznecov} provided the local bound in general
\begin{align}\label{trivial bound}
    \int_Y b(x) \psi(x)dx\ll \lambda^{(n-d-1)/2}
\end{align}
under certain conditions on $Y$, which are satisfied if its self-intersections are clean.  Here, $b \in C^\infty_c(Y)$ is a fixed cutoff function, and the notation $A\ll B$ will mean that there is a positive constant $C$
such that $|A| \leq C B$. (See Section \ref{notation} for the notation in this paper.)
See also Good \cite{good1983local} and Hejhal \cite{hejhal1982certaines} in the case where $X$ is a compact hyperbolic surface, and $Y$ is a closed geodesic, and see also Reznikov \cite{reznikov2015uniform}.

Let $\psi_j$, $j\ge0$, form an orthonormal basis of Laplace-Beltrami eigenfunctions
with frequencies $\lambda_j\geq0$, i.e. $\Delta\psi_j+\lambda_j^2\psi_j=0$.
Zelditch obtained the bound \eqref{trivial bound} by proving the Kuznecov sum
formula
\begin{align}\label{Zelditch's sum formula}
    \sum_{\lambda_j\leq\lambda}\left| \int_Y b(x) \psi_j(x)dx \right|^2 = C \lambda^{n-d} + O(\lambda^{n-d-1}),
\end{align}
where the constant $C$ depends on $X$, $Y$, and $b$, and is positive if $b$ is nonzero.
The Kuznecov sum formula \eqref{Zelditch's sum formula} implies that the bound \eqref{trivial bound} is sharp when $X$ is the $n$-dimensional sphere equipped with the standard metric. For instance, see \cite[p. 1305]{wyman2019looping}.

When the underlying manifold $X$ is of negative curvature, it is expected that the bound \eqref{trivial bound} can be strengthened.
For instance, Chen and Sogge \cite{chen2015integrals} proved the bound $o(1)$ for the geodesic periods of eigenfunctions on a compact surface of negative curvature.
Moreover, an improvement of $(\log\lambda)^{-1/2}$ to \eqref{trivial bound} can be proved in some cases.
For instance,  Sogge, Xi, and Zhang \cite{sogge2017geodesic} and Wyman \cite{wyman2021period} provided the bound $O((\log\lambda)^{-1/2})$ when $X$ is of negative curvature and $Y$ is a totally geodesic hypersurface.
See also \cite{wyman2021period} for the results when $Y$ has codimension at least 2.

In this paper, we let $X$ be a compact arithmetic hyperbolic $3$-manifold arising from a quaternion division algebra over a number field with one complex place. We let $\psi$ be a Hecke-Maass form on $X$. This is an eigenfunction of the Laplace-Beltrami operator and all unramified Hecke operators. We shall always assume $\psi$ to be $L^2$-normalized. In this context, it is more natural to let $\lambda\in\BC$ be the spectral parameter of $\psi$, that is, $\Delta \psi +(1+\lambda^2)\psi=0$. As we are considering large-eigenvalue asymptotics, we will also assume that $\lambda\in\BR$ and $\lambda>0$ to exclude exceptional eigenvalues.
Let $Y$ be a totally geodesic surface in $X$, and let $b\in C_c^\infty(Y)$ be a smooth compactly supported function on $Y$. We define
\begin{align*}
    P_Y(\psi,b) = \int_Y b(x)\psi(x) dx.
\end{align*}
As $Y$ is totally geodesic, its self-intersections are clean, and (\ref{trivial bound}) gives  $P_Y(\psi,b)  \ll 1$. 
Our main result is that the local bound $P_Y(\psi,b)  \ll 1$ can be improved by a \textit{power saving} as follows.
\begin{theorem}\label{main}
    Let $\psi$ be a Hecke-Maass form on $X$ with spectral parameter $\lambda$. For any totally geodesic surface $Y$ in  $X$ and $b\in C_c^\infty(Y)$, we have
    \begin{align*}
        P_Y(\psi,b)  \ll _\varepsilon \lambda^{-1/74+ \varepsilon},
    \end{align*}
    where the implied constant depends on $X$, the support of $b$, and the $L^\infty$-norms of $b$ and finitely many of its derivatives.
\end{theorem}

We refer to Avakumovi{\'c} \cite{avakumovic1956eigenfunktionen}  and Levitan \cite{levitan1952asymptotic} for the Weyl law on Riemannian manifolds. Recall the global Weyl law on $X$:
\begin{align*}
   \sum_{\lambda_j\leq\lambda}1 = C\lambda^{3} + O(\lambda^{2}),
\end{align*}
where $C>0$ is a constant depending only on $X$. 
Note that by combining (\ref{Zelditch's sum formula}) with the global Weyl law, we see that the average size of $P_Y(\psi, b)$ is $\lambda^{-1}$. 
\medskip

When $Y$ is a closed congruence arithmetic surface, we expect stronger bounds for the period $P_Y(\psi,b)$ to hold based on automorphic distinction principles and period formulas.  These bounds take different forms for $b$ that lie in the span of the one-dimensional representations, or of the infinite-dimensional representations, and we discuss these cases separately. In the next few paragraphs, it is explained how the bounds $P_Y(\psi,b)\ll\lambda^{-1+o(1)}$ and $P_Y(\psi,b)\ll\lambda^{-1/2+o(1)}$ would be expected to be optimal in the infinite- and finite-dimensional cases, respectively.
\medskip

We begin with the infinite-dimensional case.  Let $F$ be a totally real field, and $E$ a quadratic extension of $F$ with exactly one complex place.  We suppose that $Y$ is constructed as an adelic quotient of a group $G$ that is a form of ${\rm PGL}_2/F$, and that $X$ is similarly constructed from the base change $G_E$.  We let $L(s, \text{As} \, \psi)$ denote the Asai $L$-function of $\psi$.  If $b = \phi$ is a Hecke--Maass form on $Y$ that generates an irreducible infinite-dimensional automorphic representation, then by \cite[Theorem 1.1]{ichino2008}, the bound $P_Y(\psi, \phi) \ll \lambda^{-1/2 + \varepsilon}$ follows from the convexity bound for $L(1/2, \text{As} \, \psi \otimes \phi)$, and $P_Y(\psi, \phi) \ll \lambda^{-1 + \varepsilon}$ follows from the Lindel\"of hypothesis.  More precisely, one proves this implication by showing that the local period integrals appearing in Ichino's formula are bounded depending only on the level of $\phi$ and $\psi$ (using bounds towards Ramanujan), and applying bounds for adjoint $L$-values from \cite[Lemma 3]{blomer2010twisted}.  Moreover, if the local period integrals are bounded away from zero then the reverse implications are true.  More generally, when $b$ lies in the span of the infinite-dimensional representations one expects that $P_Y(\psi, b) \ll \lambda^{-1/2 + \varepsilon}$, which should follow by spectrally expanding $b$ and applying the convexity bound, and one likewise expects $P_Y(\psi, b) \ll \lambda^{-1 + \varepsilon}$ assuming the Lindel\"of hypothesis.
\medskip

When $b$ lies in a one-dimensional representation, then it is expected that $P_Y(\psi,b)\neq0$ implies that $\psi$ arises from a quadratic base change. To give an example where this is known, we now let $E$ be an imaginary quadratic field and allow $X$ to be a noncompact manifold obtained as an adelic quotient of the group ${\rm PGL}(2,E)$.  Let $\psi$ be a Hecke--Maass cusp form on $X$, and let $\{ \psi_i \}$ be an orthonormal basis of such cusp forms.  If $Y$ is the surface in $X$ arising from the subgroup ${\rm PGL}(2,\BQ)$ of ${\rm PGL}(2,E)$, then $P_Y(\psi, 1)$ is the Flicker--Rallis period of $\psi$.  If we assume that $\psi$ lies in an irreducible cuspidal automorphic representation $\pi$, then it is known, by e.g. \cite[Theorem 1]{ye1989kloosterman}, that $P_Y(\psi, 1)$ can be nonzero only when $\pi$ is the base change of a cuspidal automorphic representation $\sigma$ of ${\rm GL}(2,\BQ)$ with central character $\eta$, where $\eta$ is the quadratic character associated to $E/\BQ$.  Moreover, in this case, there will be a vector in $\pi$ whose Flicker--Rallis period is nonzero.

Such a distinction principle is expected to imply the existence of infinitely many forms with $P_Y(\psi, 1) \gg \lambda^{-1/2}$.  To explain why this should hold in the case of ${\rm PGL}(2,E)$ under discussion, let us assume the asymptotic (\ref{Zelditch's sum formula}) for the noncompact quotient $X$, which implies that
\begin{equation}
\label{sum}
\sum_{ |\lambda_i - \lambda | < Q} |P_Y(\psi_i, 1)|^2 \asymp 1
\end{equation}
for $Q$ sufficiently large.  One expects there to be $\asymp \lambda$ cusp forms $\psi_i$ with $| \lambda_i - \lambda | < Q$ that are base changes from ${\rm GL}(2,\BQ)$ (where the lower bound holds after possibly increasing the level of $X$), and so if we combine (\ref{sum}) with distinction for the Flicker--Rallis period, it would show the existence of forms in any window $| \lambda_i - \lambda | < Q$ with periods satisfying $P_Y(\psi_i, 1) \gg \lambda_i^{-1/2}$.

Moreover, when $\psi$ is obtained via base change, one might hope for a special value formula that gives $P_Y(\psi, b)$ in terms of $L$-functions evaluated at 1.  For instance, when $\psi$ is a form on ${\rm PGL}(2,E)$ obtained by base change from a representation $\sigma$ of ${\rm GL}(2,\BQ)$ with central character $\eta$ as above, the results of \cite[Section 5.1]{humphries2024new} and \cite{jiang2006periods} lead us to expect a special value formula for $P_Y(\psi, 1)$ of the asymptotic form 
$$P_Y(\psi, 1) \asymp \lambda^{-1/2} \sqrt{ L(1, \Ad \, \sigma \otimes \eta) / L(1, \Ad \, \sigma) },$$
either assuming some conditions on the level of $\psi$, or with local factors inserted at the ramified places.  The papers \cite[Lemma 5.1]{humphries2024new} and \cite[Theorem 1.2.1]{jiang2006periods} give a formula of this type in the case of ${\rm PGL}(2)$ over a {\it real} quadratic field, assuming that $\psi$ has full level (although potential generalizations are discussed in \cite[Remark 1.2.2]{jiang2006periods}).  Moreover, an analogous formula for ${\rm GL}(2)$ over general fields, but with a quadratic twist inserted into the period over $Y$, is proved in \cite[Proposition 3.2]{zhang2014automorphic}.  In any case, such a period formula, when combined with the formula $L(s, \Ad \pi) = L(s, \Ad \sigma) L(s, \Ad \sigma \otimes \eta)$ and the bounds $L(1, \Ad \, \pi) = \lambda^{o(1)}$ and $L(1, \Ad \, \sigma) = \lambda^{o(1)}$ from \cite[Lemma 3]{blomer2010twisted}, would give $P_Y(\psi, 1) = \lambda^{-1/2 + o(1)}$.
\medskip

We prove Theorem \ref{main} by applying arithmetic amplification. This method was introduced by Iwaniec and Sarnak in \cite{iwaniec1995norms}, which they used to deduce a stronger bound for the sup-norms of Hecke-Maass forms on congruence arithmetic hyperbolic surfaces. 
After \cite{iwaniec1995norms}, the method of arithmetic amplification is applied to bound the sup-norms of Hecke-Maass forms in different cases.
For instance, Blomer, Harcos, Maga and Mili{\'c}evi{\'c} solved the sup-norm problem for the group $\GL(2)$ over a general number field   in \cite{blomer2019sup}, which improves upon the previous results in \cite{iwaniec1995norms,koyama1995norms,blomer2016bounds}. For higher rank groups, see e.g \cite{blomer2016subconvexity}.

Marshall applied the method of amplification to the restriction problems for compact congruence arithmetic hyperbolic surfaces $S$ in \cite{marshall2016geodesic}. For any geodesic segment $l$ of unit length, he improved the $L^2$-restriction norm of $\psi|_l$ over the local bound by Burq, G{\'e}rard and Tzvetkov \cite{burq2007restrictions}, and improved the bounds for Fourier coefficients of $\psi$ along $l$. 
Using Waldspurger’s formula and $L$-functions, Ali \cite{ali20222} proved a stronger $L^2$-bound for a Hecke-Maass cusp form on the modular surface $\SL(2,\BZ)\backslash\BH^2$ restricted to closed geodesics associated to a fundamental discriminant $D>0$.

It is hoped to apply the amplification method to restriction problems on other groups. For instance, Marshall dealt with the case of restricting an $\SL(3)$ Hecke-Maass form to a maximal flat subspace in \cite{marshall2015restrictions}.
In this paper, we consider the restriction problem of the group $\SL(2,\BC)$ to the subgroup $\SL(2,\BR)$, which can be seen as a first step towards applying the amplification method for bounding periods on more general groups.

\medskip

Finally, in this paragraph, we explain the distinction and relationship between the period bound problem considered in this paper and the closely related problem on the $L^2$-restriction norm. Let $\psi$ be an $L^2$-normalized Hecke-Maass form on the compact congruence arithmetic hyperbolic $3$-manifold $X$ with spectral parameter $\lambda>0$ as before. For simplicity, we assume $Y$ to be a closed totally geodesic surface in $X$. The local bound for the $L^2$-restriction norm is shown by Burq, G{\'e}rard and Tzvetkov \cite{burq2007restrictions}, that is $\|\psi|_Y\|_{L^2(Y)}\ll\lambda^{1/4}$. We have the spectral decomposition $L^2(Y)= \bigoplus_i \BC\phi_i$, where $\phi_i$'s are Maass forms on $Y$ with spectral parameters $\mu_i$, which form an orthonormal basis of $L^2(Y)$. By the Plancherel's Theorem, the $L^2$-restriction norm can be expressed in terms of the periods:
\begin{align*}
    \|\psi|_Y\|_{L^2(Y)}^2=\sum_i |P_Y(\psi,\phi_i)|^2.
\end{align*}
The main contribution to the $L^2$ restriction norm should come from those $P_Y(\psi,\phi_i)$ where $\mu_i$ is near $\lambda$. We let $\beta$ be any positive number satisfying $\lambda^\epsilon\leq\beta\leq\lambda$. When $|\lambda - |\mu_i||\geq\beta$, we have
\begin{align}\label{local bd for Fourier coefficient}
    P_Y(\psi,\phi_i)\ll_\varepsilon \lambda^{1/4}\beta^{-1/4}.
\end{align}
This is \eqref{trivial bound} when $\mu_i\asymp 1$. Therefore, the result in this paper, which only improves \eqref{local bd for Fourier coefficient} when $\mu_i\asymp 1$, can not be applied to improve the $L^2$-restriction norm. In \cite{hou2024restrictions} considering the $L^2$-restriction problem, we prove the bound \eqref{local bd for Fourier coefficient} and prove a power saving over the local bound $\|\psi|_Y\|_{L^2(Y)}\ll\lambda^{1/4}$.

\subsection{Outline of the proof}
We will prove Theorem \ref{main} in a similar way as the proof of \cite[Theorem 1.3]{marshall2016geodesic}. 
We shall rewrite the period integral as a pairing of the eigenfunction with the smooth cutoff function $b\in C_c^\infty(\BH^2)$ integrated along $\BH^2$,
\begin{align*}
    \langle \psi,b\rangle=\int_{\BH^2} b(z)\psi(g_0.z) dz,
\end{align*}
where $b\in C_c^\infty(\BH^2)$ and $g_0$ is the isometry mapping the standard $\BH^2$ to the surface $Y$, after passing to the quotient.

To estimate $\langle\psi,b\rangle$, we let $h_\lambda$ be an even real-variable function that is non-negative on the spectral parameters and concentrates near $\pm\lambda$.
We define $k_\lambda$ to be the inverse Harish-Chandra transform for $\BH^3$ of $h_\lambda$.
Then $k_\lambda$ is an $\SU(2)$-bi-invariant test function. Let $X=\Gamma\backslash \BH^3$ be the arithmetic congruence hyperbolic $3$-manifold. We sum the test function over the lattice $\Gamma$ to form the kernel function
\begin{align*}
    K(x,y) = \sum_{\gamma\in\Gamma} k_\lambda(x^{-1}\gamma y)
\end{align*}
on $X\times X$. By the spectral expansion for $K$, we obtain the following pretrace formula
\begin{align}\label{pretrace formula in intro}
    \sum_{i} h_\lambda(\lambda_i) \psi_i(x)\overline{\psi_i(y)}=\sum_{\gamma\in\Gamma} k_\lambda(x^{-1}\gamma y).
\end{align}
Here $\psi_i$'s are Hecke-Maass forms on $X$ with spectral parameters $\lambda_i$, which form an orthonormal basis of $L^2(X)$ and $\psi$ is one of them. 

Let us first outline how to get the local bound $\langle\psi,b\rangle\ll1$ by integrating the pretrace formula. If we integrate \eqref{pretrace formula in intro} against $b\times \bar{b}$ along  $g_0\BH^2\times g_0 \BH^2$, we obtain
\begin{align}\label{integrated pretrace formula in intro}
    \sum_{i} h_\lambda(\lambda_i) |\langle \psi_i,b\rangle|^2=\sum_{\gamma\in\Gamma} \iint_{\BH^2}b(z_1)\overline{b(z_2)} k_{\lambda}(z_1^{-1}g_0^{-1}\gamma g_0z_2) dz_1 dz_2.
\end{align}
By dropping all terms on the spectral side (left-hand side of \eqref{integrated pretrace formula in intro}) except the term $|\langle \psi,b\rangle|^2$, it suffices to bound the integrals
\begin{align*}
    I(\lambda,g) :=\iint_{\BH^2}b(z_1)\overline{b(z_2)} k_{\lambda}(z_1^{-1} g z_2) dz_1 dz_2,\quad\text{ for } g\in \SL(2,\BC).
\end{align*}
We may assume that $\Gamma$ is torsion-free and the supports of $b$ and $k_\lambda$ are sufficiently small so that only $I(\lambda,e)$ contributes to the geometric side (right-hand side of \eqref{integrated pretrace formula in intro}). Therefore, the bound $I(\lambda,e)\ll1$ from Proposition \ref{bound of I} will imply the local bound $\langle\psi,b\rangle\ll1$.

We now give an outline of the amplification method. An amplifier $\cT$ is a sum of Hecke operators, which will be constructed in \eqref{amplifier 1}.
The amplifier $\cT$ is chosen so that its eigenvalue on $\psi$ is large, and it is expected to have small eigenvalues on the remaining $\psi_i$, which helps us to prove a power saving.
We apply the amplifier $\cT$ and then integrate \eqref{pretrace formula in intro} against $b\times \bar{b}$ along  $g_0\BH^2\times g_0 \BH^2$ to obtain
\begin{align}\label{ampl ineq in intro}
    |\langle \cT\psi,b\rangle|^2\ll \sum_{i} h_\lambda(\lambda_i) |\langle \cT\psi_i,b\rangle|^2=\sum_{\gamma} C(\gamma)I(\lambda,g_0^{-1}\gamma g_0).
\end{align}
The geometric side is a weighted sum of $I(\lambda,g_0^{-1}\gamma g_0)$ for $\gamma$ running through the translations appearing in $\cT*\cT^*$ with coefficients $C(\gamma)$ determined by $\cT$. The amplification inequality \eqref{ampl ineq in intro} is proved in Proposition \ref{amplification inequality}.
To apply \eqref{ampl ineq in intro}, we must bound the geometric side, which requires solving a counting problem (Proposition \ref{point counting prop}) and a quantitative control on the decay of $I(\lambda,g)$ when $g$ is (in a certain sense) away from stabilizing $\BH^2$ (Proposition \ref{bound of I}). By combining these two ingredients, we prove Theorem \ref{main} in Section \ref{sec bd for period}.

The counting problem is estimating the number of times a Hecke operator maps $\BH^2$ back close to itself. We prove this Hecke return estimate in Section \ref{counting}. We consider a 4-dimensional representation of $\SL(2,\BC)$, coming from the standard representation of $\mathrm{SO}(3,1)$. It will be shown in Lemma \ref{fixing a vector} that if $\gamma$ maps $\BH^2$ back close enough to itself then $\gamma$ fixes a rational line in the representation. We can get a contradiction from our construction of the Hecke operators appearing in the amplifier to conclude that no such $\gamma$ exists.

The analytic problem for bounding $I(\lambda,g)$ in Proposition \ref{bound of I} is the main part of this paper. The proof occupies Section \ref{integral}. We give a more detailed outline in Section \ref{outline of the section}. We unfold the integral $I(\lambda,g)$ using the inverse Harish-Chandra transform and the integral representation \eqref{eq:HC K int formula} for the spherical function. This leads to bound an oscillatory integral over $\SU(2)\times\BH^2\times g\BH^2$. The estimate of the oscillatory integral is based on the method of stationary phase.



\subsection{Acknowledgements}
The author would like to thank his advisor Simon Marshall for suggesting this
problem and providing helpful advice about this paper.  He would like to thank Peter Humphries, Emmett Wyman, and Cheng Zhang for helpful comments. He would also like to thank the referee for a careful reading and many helpful comments. The author was supported by NSF grants DMS-1902173 and DMS-1954479.

\section{Notation}\label{notation}
Throughout the paper, the notation $A\ll B$ will mean that there is a positive constant $C$
such that $|A| \leq C B$, and $A\asymp B$ will mean that there are positive constants $C_1$ and $C_2$ such
that $C_1B \leq A \leq C_2B$.

\subsection{Quaternion algebras and adelic groups}
Let $F$ be a totally real number field with the ring of integers $\cO$. We denote the norm of an ideal $\fn$ of $\cO$ by $\operatorname{N}(\fn)$. If $v$ is a finite place of $F$, $F_v$ is the local field of $F$ at $v$ with the ring of integers $\cO_v$, $\varpi_v$ and $q_v$ denote a uniformizer and the order of the residue field. Let $E$ be a quadratic extension of $F$ with exactly one complex place $w_0$, and let $v_0$ be the place below $w_0$. Let $\cO_E$ be the ring of integers of $E$. We denote by $\operatorname{N}_E$ the norm of ideals in $\cO_E$. We let $\bar{\cdot}$ denote the conjugation of $E$ over $F$, and denote the trace and norm from $E$ to $F$ by $\operatorname{Tr}_{E/F}$ and $\operatorname{N}_{E/F}$, respectively. We denote the rings of adeles of $F$ and $E$ by $\BA$ and $\BA_E$, and the ring of finite adeles of $F$ by $\BA_f$.
Let $|\cdot|_v$ be the absolute value on $F_v$ for any place $v$ of $F$, and let $|\cdot| = \prod_v |\cdot|_v$ be the modulus on $\BA$.

Let $D$ be a quaternion division algebra over $F$ that is ramified at every real place except at $v_0$, and let $B = D\otimes_F E$. We assume that $B$ is also a division algebra. We denote the standard involution on $D$ and $B$ by $\iota$, and also use $\bar{\cdot}$ to denote the conjugation of $B$ over $D$. Let $\operatorname{trd}_B(x) = x+\iota(x)$ and $\operatorname{nrd}_B(x)=x\iota(x)$ denote the reduced trace and the reduced norm on $B$. We let $D_0$ denote the $-1$-eigenspace for $\iota$ acting on $D$. Choose a nonzero element $\eta\in E$ with $\operatorname{Tr}_{E/F}(\eta)=0$, and define the $F$-subspace $V$ of $B$ by
\begin{align*}
    V = \eta D_0+F= \{ x\in B|\bar{x}=\iota(x)\}.
\end{align*}
Then it may be checked that $V$ is stable under the action of $B$ given by $b\cdot x = bx\overline{\iota(b)}$. Hence, $V$ is an $F$-linear representation of $B^\times$ that we denote by $\rho$. We define a $F$-valued bilinear form $\langle\cdot,\cdot\rangle$ on $V$ by $\langle x,y\rangle = \operatorname{Tr}_{E/F}\left(\operatorname{trd}_B\left(x\iota(y)\right)\right)$, which is nondegenerate and satisfies $\langle\rho(b)x,\rho(b)y \rangle = \operatorname{N}_{E/F}(\operatorname{nrd}_B(b))\langle x,y\rangle$. Moreover, $\langle \cdot,\cdot\rangle$ has signature $(3,1)$ at the place $v_0$.

We denote the groups of reduced norm $1$ elements in $D$ and $B$ by $D^1$ and $B^1$, and denote the multiplicative group by $D^\times$ and $B^\times$. We let $\bG$ and $\bH$ be the algebraic groups over $F$ such that $\bG(F) = B^1$, $\bH(F) = D^1$.
We denote $D\otimes_F F_v$, $B\otimes_F F_v$,  $\bG(F_v)$ and $\bH(F_v)$  by $D_v$, $B_v$, $G_v$ and $H_v$. We still denote by $\rho$ the representations of $\bG$ and $\bH$ obtained by restricting the action of $B$ on $V$, which is algebraic over $F$.

If $\bv_1$ is the vector in $V$ corresponding to $1\in B$, it may be seen that $\bH$ is the stabilizer of $\bv_1$ in $\bG$, through the representation $\rho$.
At $v_0$, we have $G_{v_0} \simeq \SL(2,\BC)$, $H_{v_0}\simeq \SL(2,\BR)$, and $V_{v_0} = V \otimes_F F_{v_0}$, where
\begin{align*}
    V_{v_0} \simeq \left.\left\{ \begin{pmatrix}
        a&b\\c&d
    \end{pmatrix} \in M_2(\BC) \right|   d = \bar{a},\text{ and }b,c\in i\BR \right\}
\end{align*}
with
\begin{align}\label{representation rho}
    \rho(g) A = g A \bar{g}^{-1}, \quad g\in\SL(2,\BC)\text{ and }A\in V_{v_0}.
\end{align}
Hence, the complexification $V_{v_0}\otimes\BC$ is the representation of $\SL(2,\BC)$ acting on the vector space of $2$-by-$2$ matrices $M_2(\BC)$ given by the same formula as (\ref{representation rho}). We will implicitly make these identifications later.
Alternatively, the representation $(\rho,V_{v_0})$ is isomorphic to the representation of $\SL(2,\BC)$ obtained by composing the quotient map to $\SO(3,1)^\circ$ with the standard representation of $\SO(3,1)$.
Here $\SO(3,1)^\circ$ is the identity component of the Lie group $\SO(3,1)$. For simplicity, we denote $G_0=G_{v_0}$, $H_0 = H_{v_0}$,  $K_0=K_{v_0}$ and $V_0 =V_{v_0}$.

Let $H^\prime$ be the normalizer of $H_0$ in $G_0$, which has two connected components with the identity component $H_0$. In fact,
\begin{align*}
    H^\prime = H_0 \cup \begin{pmatrix}0&i\\i&0\end{pmatrix} H_0. 
\end{align*}
If we consider the representation $(\rho,V_{v_0})$ of $G_0$, then $H^\prime$ is the subgroup of $G_0$ fixing the line spanned by the vector $\bv_1$.
Moreover, the identity component $H_0$ of $H^\prime$ acts on $\bv_1$ trivially, and the action of the non-identity component on $\bv_1$ is by multiplying $-1$.

Let $\cO_B\subset B$ be a maximal $\cO$ order, and let $S$ be a finite set of places of $F$ containing all infinite places and all places where $D$ ramifies. We choose a compact subgroup $K = \prod_v K_v$ of $\bG(\BA)$ as follows.
We choose $K_{v_0} = \SU(2)$, and $K_v = B^1_v$ for all other real places $v$.
For finite places in $S$ we let $K_v\subset G_v$ be any subgroup that stabilizes $\cO_{B,v} = \cO_B \otimes _\cO \cO_v$, and for other finite places we let $K_v = B^1_v\cap \cO_{B,v}$. Let $\cL = V \cap \cO_B$ be a lattice in $V$, where we identify $V$ as a subspace of $B$.
After enlarging $S$ also to include all places of $F$ that are ramified in $E$, we may choose an isomorphism $\alpha$ from $B_v$ to the product of two $2$ by $2$ matrix algebras $M_2(F_v)\times M_2(F_v)$ for all $v\notin S$ that is split in $E$.  The isomorphism $\alpha$ satisfies: if $v$ splits,
\begin{itemize}
    \item $\alpha(D_v) = \{ (T,T)|T\in M_2(F_v) \}$, 
    \item $\alpha(V_v)=\{ (T,\iota(T))|T\in M_2(F_v) \}$,
    \item $\alpha(K_v) = \SL(2,\cO_v)\times\SL(2,\cO_v)$,
    \item $\alpha(\cL_v)=\{ (T,\iota(T))|T\in M_2(\cO_v) \}$,
\end{itemize}
and, moreover, here the conjugation $\bar{\cdot}$ is identified with the map switching the two factors.
Let $\sP$ be the set of primes $v\notin S$ that split in $E$. We shall implicitly make the identification $\alpha$ at places in $\sP$. For $v\in \sP$, we write $G_v = \SL(2,F_v)\times\SL(2,F_v)$, and  $K_v = K_{v,1}\times K_{v,2} =\SL(2,\cO_v)\times\SL(2,\cO_v) $.

\subsection{Lie groups and algebras}

Let $A$ be the connected component of the diagonal subgroup of $\mathrm{SL}(2,\BR)$ with parametrization
\begin{align*}
    a(t) = \begin{pmatrix}e^{t/2} &0 \\0 &e^{-t/2} \end{pmatrix}, \quad t\in\BR.
\end{align*}
Let
\begin{align*}
    N = \left\{ \begin{pmatrix}1 &z \\0 &1 \end{pmatrix} \Big| z\in\BC \right\}\quad \text{and}\quad N_0 = \left\{ \begin{pmatrix}1 &x \\0 &1 \end{pmatrix} \Big| x\in\BR \right\}
\end{align*}
be the unipotent subgroups.
We denote the Lie algebras of $K_{0}$, $A$, $N$ and $N_0$ by $\fk$, $\fa$, $\fn$ and $\fn_0$. We write the Iwasawa decomposition of $G_0 = NAK_0$ as
\begin{align*}
    g = n(g)\exp\left(A(g)\right) \kappa(g) = \exp\left( N(g)\right)\exp\left(A(g)\right) \kappa(g).
\end{align*}
We define
\begin{align*}
    H = \begin{pmatrix}1/2 &0 \\0 &-1/2 \end{pmatrix} \in \fa,\quad X = \begin{pmatrix}0 &1 \\0 &0 \end{pmatrix} \in\fn.
\end{align*}
We identity $\fa$ with the real line $\BR$ under the map $H\mapsto 1$ and consider $A(g)$ as a function $A:G_0\to\BR$ under this identification, and we obtain the identifications $\fn\simeq \BC$ and $\fn_0\simeq\BR$ by sending $X$ to $1$. We identify the dual space $\fa^*$ of $\fa$ as $\BR$ by sending the root $tH\mapsto t$ to $1$. Under these identifications, the pairing between $\fa$ and $\fa^*$ is the multiplication in $\BR$.
The parametrization of $A$ can be written as the homomorphism $a:\BR \rightarrow A$ by $a(t) = \exp(tH)$. We define $n:\BC \to N$ by $n(z) = \exp(zX)$. The restriction of $n$ to $\BR\to N_0$ will still be denoted by $n$. We denote by $\fa^+ \simeq \BR_{>0}$ the positive Weyl chamber, and $A^+$ the image $\exp(\fa^+) = a(\BR_{>0})$. We denote the diagonal subgroup of $K_0$ by $\mathrm{U}(1)$, i.e.,
\begin{align}\label{defn of U(1)}
     \mathrm{U}(1)=\left. \left\{ \begin{pmatrix}
    e^{it}&\\&e^{-it}
\end{pmatrix} \right|\,t\in\BR\right\}.
\end{align}
Moreover, $\mathrm{U}(1)$ is the centralizer of $A$ in $K_0$.

We equip $\mathfrak{sl}(2,\BC)$ with the norm 
\begin{align*}
    \|\cdot \|:\begin{pmatrix}Z_1&Z_2\\Z_3&-Z_1\end{pmatrix}\mapsto \left(\|Z_1\|^2+\|Z_2\|^2+\|Z_3\|^2 \right)^{1/2}.
\end{align*}
This norm induces a left-invariant metric on $G_0$, denoted by $d$.

We define a Haar measure $dg$ on $G_0$ through the Iwasawa decomposition $G_0=NAK_0$. Namely, if $g=n(z)a(t)k$ then $dg=e^{-2t}dzdtdk$. Here $dz,dt$ are the standard measures on $\BC$ and $\BR$ as Euclidean spaces, and $dk$ is the probability Haar measure on $K_0$.

\subsection{Hecke algebras}
For any continuous function $f$ on $\bG(\BA)$, we define $f^*(g) = \overline{f(g^{-1})}$. We define $\cH_f = \bigotimes_{v<\infty}^\prime \cH_v$ to be the convolution algebra of smooth functions on $\bG(\BA_f)$ that are compactly supported and bi-invariant under $K_f = \prod_{v<\infty}^\prime K_v$, and $\cH_v$ denote the space of smooth, compactly supported functions on $G_v$ that are bi-invariant under $K_v$. 
If $v\in\sP$ and $a_1,a_2\in\BZ$, we define $K_v(a_1,a_2)$ to be the double coset
\begin{align*}
    K_v(a_1,a_2) = K_{v,1}(a_1)\times K_{v,2}(a_2), 
\end{align*}
where $K_{v,i}(a_i) = K_{v,i}\begin{pmatrix}\varpi_v^{a_i}&\\&\varpi_v^{-a_i}\end{pmatrix}K_{v,i}$.
We let
$T_v(a)$ be the characteristic function of $K_v(a,0)$. 
Given an ideal $\fn\subset\cO$, suppose that $\fn$ is only divisible by prime ideals in $\sP$.
We define the double coset in $\bG(\BA_f)$
\begin{align*}
    K(\fn) = \prod_{v\in\sP} K_v(\operatorname{ord}_{v}(\fn),0) \times \prod_{v\notin\sP} K_v.
\end{align*}
In this paper the notation $K(\fn)$ is used non-standardly. We use this notation to denote the above Hecke double coset but not a principal congruence subgroup of $K$.
The action of $\phi\in\cH_f$ on an automorphic function $f$ on $\bG(F)\backslash \bG(\BA)$ is given by the right regular action
\begin{align*}
    [\phi f](x) = \int_{\bG(\BA_f)}\phi(g)f(xg) dg.
\end{align*}
Here, we use the Haar measures $dg_v$ on $G_v$, which are normalized so that $K_v$ has unit volume.

\subsection{Arithmetic manifolds and Hecke-Maass forms}

Define $X = \bG(F)\backslash\bG(\BA)/ K$, which is a compact connected hyperbolic $3$-manifold. We let $\Omega = \prod_{v}\Omega_v\subset\bG(\BA)$ be a compact set containing a fundamental domain for $\bG(F)\backslash \bG(\BA)$.
The universal cover of $X$ is the hyperbolic $3$-space $\BH^3$, which can be identified with the quotient $G_0/ K_{0} = \SL(2,\BC)/\SU(2)$.
We will use the upper half 3-space model of $\BH^3$ and the upper half plane model of $\BH^2$, i.e., 
\begin{align*}
    \BH^3 = \{(z,t) \,|\, z\in\BC,\, t\in\BR_{>0} \}\quad\text{and}\quad\BH^2 = \{(x,t)\,|\,x\in\BR,\,t\in\BR_{>0} \}.
\end{align*}
The embedding $\BH^2\subset\BH^3$ is identified with the natural embedding $\SL(2,\BR)/\SO(2)\subset\SL(2,\BC)/\SU(2)$. For $g =  \begin{pmatrix}a &b \\c &d \end{pmatrix} \in  \SL(2,\BC)$ and $(z,t)\in\BH^3$, we have the following formula:
\begin{align}\label{Poincare}
    g.(z,t) = \left(\frac{(az+b)\overline{(cz+d)}+a\overline{c}t^2}{|cz+d|^2+|c|^2t^2},\frac{t}{|cz+d|^2+|c|^2t^2} \right).
\end{align}
Note that $K_0$ is the stabilizer of $o=(0,1)$ in $G_0$ and $H^\prime$ is the stabilzer of $\BH^2$ in $G_0$.
We denote by 
\begin{align}\label{defn of l}
    l=A.o=\{(0,t)\in\BH^3\,|\, t>0 \}
\end{align}
the vertical geodesic through $o$ in $\BH^3$. When $g\in\SL(2,\BC)$ is acting on sets, e.g. $l$, $\BH^2$, we will simply write them as $gl$, $g\BH^2$.

By a \textit{Hecke-Maass form} on $X$ we mean an eigenfunction of the Laplacian $\Delta$ and the Hecke algebras $\cH_v$ for all $v\notin S$. In fact, the proof of Theorem \ref{main} only uses Hecke operators over places in $\sP$. We let $\psi\in L^2(X)$ be a Hecke-Maass form and let $\lambda$ be its spectral parameter, so that
\begin{align*}
    \Delta\psi+(1+\lambda^2)\psi=0.
\end{align*}
We assume that $\|\psi\|_2=1$ with respect to the hyperbolic volume on $X$ and $\lambda>0$. Note that because $\Delta$ and $T_v\in\cH_v$, $v\notin S$, are self-adjoint, we may assume that $\psi$ is real-valued.

\subsection{Harish-Chandra transforms}
For $s\in\BC$ and $x\in\BH^3$, we denote by $\varphi_s(x)$ the spherical function on $\BH^3$ with the spectral parameter $s$.
We shall also think of $\varphi_s$ as a $K_0$-bi-invariant function on $G_0$ by $\varphi_s(g)=\varphi_s(g.o)$ for $g\in G_0$.
From e.g. \cite[Ch. \RNum{4}, Theorem 4.3]{helgason1984groups}, we have the following integral formula for the spherical function:
\begin{align}\label{eq:HC K int formula}
    \varphi_s(g)=\int_{K_0}\exp({(1+is)A(kg)})dk,\quad g\in G_0.
\end{align}
Let $f\in C_c^\infty(\BH^3)$ be left $K_0$-invariant. Its Harish-Chandra transform $\widehat{f}$ is defined by the integral:
\begin{align*}
    \widehat{f}(s)=\int_{G_0} f(g.o)\varphi_{-s}(g) dg,\quad s\in\BC.
\end{align*}
We denote by $d\mu(s)$ the Plancherel measure for $\BH^3$ so that the inversion formula holds. Namely,
\begin{align}\label{eq: HC inversion}
    f(x)&=\frac{1}{2}\int_{-\infty}^\infty \widehat{f}(s)\varphi_s(x)d\mu(s)=\int_{0}^\infty \widehat{f}(s)\varphi_s(x)d\mu(s),\quad x\in\BH^3.
\end{align}
By the formula of Gindikin-Karpelevic \cite[Ch. \RNum{4}, Theorem 6.14]{helgason1984groups}, it may be seen that $d\mu(s)/ds=c\cdot s^2$ for some nonzero constant $c$.

\section{Amplification}\label{section of amplify}

This section gives an amplification inequality that we derive from the pretrace formula. We will obtain our bound on periods after stating the estimates of Hecke returns proven in Section \ref{counting} and the estimates of oscillatory integrals established in Section \ref{integral}.

Let $g_0\in \Omega_{v_0}$, and $b\in C_c^\infty(\BR^2)$. We shall study the integral
\begin{align*}
    \langle \psi,b\rangle=\int_{\BR^2} b(x,t)\psi(g_0n(x)a(t)) e^{-t}dxdt.
\end{align*}
Note that $e^{-t}dxdt$ is the measure associated to the hyperbolic metric on $\BH^2$ after the change of variables from $\BR^2$ to $\BH^2$ by $(x,t)\mapsto (x,e^t)$.
By using a partition of unity on $Y$, to prove Theorem \ref{main}, it suffices to prove
\begin{align}\label{main'}
    \langle \psi,b\rangle \ll_\varepsilon \lambda^{-1/74+\varepsilon},
\end{align}
with the implied constant depending on $X$, the support of $b$ and the $L^\infty$-norms of $b$ and finitely many of its derivatives.

\subsection{An amplification inequality and amplifiers}
We fix a real-valued function $h\in C^\infty(\BR)$ of Paley-Wiener type that is nonnegative and satisfies $h(0)=1$. Define $h_\lambda^0(s) = h(s-\lambda)+h(-s-\lambda)$, and let $k_\lambda^0$ be the $K_0$-bi-invariant function on $\BH^3$ with the Harish-Chandra transform $h^0_\lambda$. The Paley-Wiener theorem \cite[Ch. \RNum{4}, Theorem 7.1]{helgason1984groups} implies that $k_\lambda^0$ is of compact support that may be chosen arbitrarily small. Define $k_\lambda = k_\lambda^0*k_\lambda^0$, which has the Harish-Chandra  transform $h_\lambda = (h_\lambda^0)^2$. If $g\in G_{0}$, we define
\begin{align*}
    I(\lambda,g) = \int_{\BR^4}b(x_1,t_1)\overline{b(x_2,t_2)} k_{\lambda}(a(-t_1)n(-x_1)gn(x_2)a(t_2)) e^{-t_1-t_2}   dx_1 dt_1 dx_2  dt_2.
\end{align*}
We are also free to shrink the support of $b$ because we are only interested in upper bounds. Hence, we can assume that the supports of $b$ and $k_\lambda$ are small enough so that $I(\lambda,g)=0$ unless $d(g,e)\leq 1$, and denote this compact subset by $\cB\subset G_{0}$. 
The inequality that we shall use is the following.
\begin{proposition}\label{amplification inequality}
    Suppose $\cT\in\bigotimes_{v\notin S}^\prime \cH_v$. We have
    \begin{align}\label{amplification inequality eqn}
        \left| \langle \cT\psi,b\rangle\right|^2\ll \sum_{\gamma\in \bG(F)} \left| [\cT*\cT^*](\gamma)I(\lambda,g_0^{-1}\gamma g_0) \right|.
    \end{align}
\end{proposition}
\begin{proof}
    Consider the function 
    \begin{align*}
        K(x,y)= \sum_{\gamma\in\bG(F)} k_\infty [\cT*\cT^*](x^{-1}\gamma y)
    \end{align*}
    on $\bG(F)\backslash\bG(\BA)\times\bG(F)\backslash\bG(\BA)$, where $k_\infty$ is a compactly supported and $K_\infty$-bi-invariant function on $G_\infty$ defined by $k_\infty(x_\infty) = k_\lambda(x_{v_0})$. The spectral decomposition of $L^2(X)$ is
    \begin{align*}
        L^2(X) = \bigoplus_i \BC\psi_i,
    \end{align*}
    where $\psi_i$'s are Hecke-Maass forms on $X$ with spectral parameters $\lambda_i$, which form an orthonormal basis of $L^2(X)$ and $\psi$ is one of them. Then by \cite{selberg1956harmonic}, the integral operator acts on Hecke-Maass forms $\psi_i$ as
    \begin{align*}
        \int_{\bG(F)\backslash\bG(\BA)} K(x,y)\psi_i(y) dy &= \int_{\bG(\BA)} k_\infty[\cT*\cT^*](x^{-1}y)\psi_i(y)dy = h_\lambda(\lambda_i)[{\cT^*} \cT\psi_i](x).
    \end{align*}
    Hence, $K(x,y)$ has a spectral expansion
    \begin{align*}
        K(x,y) = \sum_i h_\lambda(\lambda_i)[\cT\psi_i](x)\overline{[\cT\psi_i](y)}.
    \end{align*}
    If we integrate it against the cutoff function on $g_0\SL(2,\BR) \times g_0\SL(2,\BR)$, we obtain
    \begin{align*}
         \sum_i h_\lambda(\lambda_i)\left| \langle \cT\psi,b\rangle\right|^2&=\int_{\SL(2,\BR)}\int_{\SL(2,\BR)} b(x)\overline{b(y)}K(g_0x,g_0y) dxdy \\ &=\sum_{\gamma\in \bG(F)} [\cT*\cT^*](\gamma)\int_{\SL(2,\BR)}\int_{\SL(2,\BR)} b(x)\overline{b(y)}k_\lambda(x^{-1}g_0^{-1}\gamma g_0y) dxdy. 
    \end{align*}
    Since we have $h_\lambda(\lambda_i)\geq 0 $ for all $i$, dropping all terms but $\psi$ completes the proof.
\end{proof}

To apply the above inequality, we need first to construct an element $\cT_v\in\cH_v$ that will form part of the amplifier $\cT$,
where $v\in\sP$ is a split finite place. 
Note that the Hecke operator $T_v(1)$ (resp. $T_v(2)$) corresponds to the operator summing over the set of nodes at distance 2 (resp. 4) from the given node in the Bruhat-Tits tree for $\mathrm{PGL}(2,F_v)$.
Then an elementary computation gives the relations in $\cH_v$:
\begin{align}\label{Hecke relation}
    T_v(1)*T_v(1) = q_v(q_v+1) +(q_v-1)T_v(1)+ T_v(2).
\end{align}
Note that we have identified $K_v$ with $\SL(2,\cO_v)\times\SL(2,\cO_v)$, and $T_v(1)$ only sees the first coordinate, so it may be seen that \eqref{Hecke relation} can also be obtained by applying the classical relation \cite[\S 1.4, (4.12)]{bump1998automorphic}.
If we define the real numbers $\tau_v(1)$ and $\tau_v(2)$ by
\begin{align*}
    T_{v}(1)\psi = \tau_v(1)q_v \psi,\quad T_{v}(2)\psi = \tau_v(2)q_v^2\psi,
\end{align*}
then (\ref{Hecke relation}) implies that we cannot have both $|\tau_v(1)|\leq 1/4$ and $|\tau_v(2)|\leq1/4$.
We define
\begin{align}\label{amplifier 1}
    T_v = \begin{cases} {T_{v}(1)}/\tau_{v}(1)q_v\quad\text{ if }|\tau_{v}(1)|>1/4,\\
    {T_{v}(2)}/\tau_{v}(2)q_v^2\quad\text{ otherwise. }
    \end{cases}
\end{align}
It follows that $T_v\psi = \psi$ for all $v\notin S$.
Note that $|K_v(1,0)/K_v|\asymp q_v^2$ and $|K_v(2,0)/K_v|\asymp q_v^4$, so $\|T_v(1)\|_{L^1}=\|T_v(1)\|_{L^2}^2\asymp q_v^2$ and $\|T_v(2)\|_{L^1}=\|T_v(2)\|_{L^2}^2\asymp q_v^4$. Hence,
\begin{align}\label{eq: L1 bd for Tv}
    \|T_v\|_{L^1}\ll q_v^2,
\end{align}
and
\begin{align}\label{eq: L2 bd for Tv}
    \|T_v\|_{L^2}\ll 1.
\end{align}

\subsection{Bounds for periods}\label{sec bd for period}

We let $g\in\Omega_{v_0}$,  $\fn\subset\cO$ and $\delta>0$, and suppose that $\fn$ is only divisible by prime ideals in $\sP$. 
We define the set 
\begin{align*}
    M(g,\delta,\fn) = \left\{ \gamma\in\bG(F)\cap K(\fn) \,|\, d(g^{-1}\gamma g,e)\leq 1, d(g^{-1}\gamma g,H^\prime)\leq \delta\right\}.
\end{align*}
The cardinality of $M(g,\delta,\fn)$ describes how many times the Hecke operators map $g\BH^2$ close to itself. We can control the size of $M(g,\delta,\fn)$ as follows.
\begin{proposition}\label{point counting prop}
    There exists a constant $C>0$ with the following property. Let $g\in\Omega_{v_0}$ and $X>0$. There exist a set $\sP_0\subset\sP$, depending on $g$ and with $\#(\sP\backslash\sP_0)\ll\log X$, and with the following property. If $\delta<CX^{-8}$, and the ideal $\fn$ is divisible only by primes in $\sP_0$, and satisfies $\operatorname{N}(\fn)<X$, $\fn\neq\cO$, then $M(g,\delta,\fn) = \emptyset$.
\end{proposition}

To estimate the right-hand side of the amplification inequality in Proposition \ref{amplification inequality}, we shall estimate the integral $I(\lambda,g)$ as follows.
\begin{proposition}\label{bound of I}
    We have $I(\lambda,g)\ll \left(1+\lambda\,d(g,H^\prime)\right)^{-1}$ if $g\in G_0$ and $d(g,e)\leq 1$. Here the implied constant depends on the support of $b$ and the $L^\infty$-norms of $b$ and finitely many of its derivatives.
\end{proposition}

We now can prove our main result.

\begin{proof}[Proof of Theorem \ref{main}]
Let $M\geq 1$ be a parameter to be chosen later and let $X= M^{4}$. Let $C>0$ be the constant appearing in Proposition \ref{point counting prop}. Applying Proposition \ref{point counting prop} to $g=g_0$ we get the set of primes $\sP_0$. We define $\sP_M = \{ v\in\sP_0 | M/2\leq q_v\leq M\}$, and define
\begin{align*}
    \cT_M = \sum_{v\in\sP_M}  T_v ,
\end{align*}
where $T_v$'s are defined by (\ref{amplifier 1}). This choice of $\cT_M$ satisfies
\begin{align*}
    |\langle\cT_M\psi,b \rangle| = \#\sP_M \cdot |\langle\psi,b \rangle| \gg M^{1-\varepsilon} |\langle\psi,b \rangle|.
\end{align*}
By Proposition \ref{amplification inequality}, it follows that
\begin{align}\label{amplified inequality 1}
     |\langle\psi,b \rangle|^2 \ll M^{-2+\varepsilon} \sum_{\substack{\gamma\in \bG(F)\\ \gamma\in g_0 \cB g_0^{-1}}} \left| [\cT_M*\cT_M^*](\gamma)I(\lambda,g_0^{-1}\gamma g_0) \right|.
\end{align}
Recall that we define $\cB$ to be the compact set in $G_0$ consisting of $g$ with $d(g,e)\leq 1$.
We choose $\delta = CX^{-8} = CM^{-32}$, and break the sum in \eqref{amplified inequality 1} into those terms with $d(g_0^{-1}\gamma g_0,H^\prime)\leq \delta $ and the complement, which we denote by $\cD$ so that
\begin{align*}
    \cD =\{\gamma\in \bG(F)|\, g_0^{-1}\gamma g_0\in \cB\text{ and }d(g_0^{-1}\gamma g_0,H^\prime)> \delta  \} .
\end{align*}
To estimate the sum over $\cD$, we first use the bound $|I(\lambda,g_0^{-1}\gamma g_0) |\ll (\lambda\delta)^{-1} \ll \lambda^{-1}M^{32}$ from Proposition \ref{bound of I}. This gives
\begin{align*}
    \sum_{\gamma\in \cD} \left| [\cT_M*\cT_M^*](\gamma)I(\lambda,g_0^{-1}\gamma g_0) \right| \ll  \lambda^{-1}M^{32}\sum_{\gamma\in \bG(F)\cap g_0 \cB g_0^{-1}}\left| [\cT_M*\cT_M^*](\gamma) \right|\ll \lambda^{-1}M^{32} \|\cT_M*\cT_M^*\|_{L^1}.
\end{align*}
Note that \eqref{eq: L1 bd for Tv} gives $\|T_v\|_{L^1} \ll M^2$ so $\|\cT_M\|_{L^1} \ll M^3$. Therefore,
\begin{align*}
    \sum_{\gamma\in \cD} \left| [\cT_M*\cT_M^*](\gamma)I(\lambda,g_0^{-1}\gamma g_0) \right| \ll \lambda^{-1}M^{38}.
\end{align*}
We next estimate the sum over the complement of $\cD$. We first use the bound $|I(\lambda,g_0^{-1}\gamma g_0) |\ll 1$ from Proposition \ref{bound of I} to obtain
\begin{align}\label{ine 3}
    \sum_{\substack{\gamma\in \bG(F)\\ \gamma\in g_0 \cB g_0^{-1}\\d(g_0^{-1}\gamma g_0,H^\prime)\leq \delta}}  \left| [\cT_M*\cT_M^*](\gamma)I(\lambda,g_0^{-1}\gamma g_0) \right| \ll \sum_{\substack{\gamma\in \bG(F)\\ \gamma\in g_0 \cB g_0^{-1}\\d(g_0^{-1}\gamma g_0,H^\prime)\leq \delta}} \left| [\cT_M*\cT_M^*](\gamma)\right|.
\end{align}
We next expand $\cT_M*\cT_M^*$ as a sum
\begin{align*}
    \cT_M*\cT_M^* = \sum_{\fn\subset\cO} a_\fn 1_{K(\fn)}
\end{align*}
for some constants $a_\fn$. Our choice of $X$ and $\delta$ means that we may apply Proposition \ref{point counting prop} to show that for any $\fn\neq \cO$ appearing in the expansion of $\cT_M*\cT_M^*$, the term $1_{K(\fn)}$ makes no contribution to the sum in the right-hand side of (\ref{ine 3}). Hence, we only need to consider the term $\fn=\cO$. By the bound \eqref{eq: L2 bd for Tv}, we have
\begin{align}\label{eq: bd for aO}
    a_\cO=[\cT_M*\cT_M^*] (e)=\|\cT_M\|_{L^2}^2= \sum_{v\in\sP_M}\|T_v \|^2_{L^2}\ll M.
\end{align}
Applying \eqref{eq: bd for aO} in \eqref{ine 3} gives
\begin{align*}
     \sum_{\substack{\gamma\in \bG(F)\\ \gamma\in g_0 \cB g_0^{-1}\\d(g_0^{-1}\gamma g_0,H^\prime)\leq \delta}} \left| [\cT_M*\cT_M^*](\gamma)I(\lambda,g_0^{-1}\gamma g_0) \right| \ll  M\sum_{\substack{\gamma\in \bG(F)\\ \gamma\in g_0 \cB g_0^{-1}}} 1_{K_f}(\gamma)\ll M.
\end{align*}
Adding our two bounds gives
\begin{align*}
    |\langle\psi,b \rangle|^2 \ll M^{-2+\varepsilon}(\lambda^{-1}M^{38}+M) = \lambda^{-1}M^{36+\varepsilon} + M^{-1+\varepsilon}.
\end{align*}
Choosing $M = \lambda^{1/37}$ gives  \eqref{main'}, which completes the proof.
\end{proof}

\section{Estimates of Hecke returns}\label{counting}

In this section, we prove Proposition \ref{point counting prop}. The following proof is taken from an unpublished note by Simon Marshall. We thank him for permission to write it here.

\subsection{Vector spaces over a number field}

We first need some definitions of adelic norms and normed vector spaces over a number field. Let $X$ be a finite-dimensional vector space over $F$.
If $v$ is a place of $F$, we let $X_v = X\otimes_F F_v$. We shall work with norms $\|\cdot\|_v$ on the spaces $X_v$, which are assumed to satisfy the following conditions. If $v$ is an infinite place, we assume that the norm comes from a positive definite quadratic or hermitian form. If $v$ is finite we assume that $\|x\|_v \in q_v^\BZ\cup\{0\}$ for all $x\in X_v$, which implies that the unit ball $L$ of $\|\cdot\|_v$ is a lattice in $X$ such that $\|x\|_v = \operatorname{min} \{|c|_v\,|\,x\in cL\}$.

By a norm on $X(\BA)$, we mean a choice of norm $\|\cdot\|_v$ on each $X_v$, with a compatibility condition that there is an $\cO$-lattice $\cL\subset X$ such that $\cL_v = \cL\otimes_\cO \cO_v$ is the unit ball of $\|\cdot\|_v$ for almost all $v$. This ensures that we  may define $\|x\|$ for $x\in X(\BA)$ by the formula $\|x\| = \prod_v\|\cdot\|_v$. We note that $\|\cdot\|$ does not satisfy the triangle inequality. We define a normed $F$-space to be a finite-dimensional vector space $X$ over $F$ with a norm on $X(\BA)$.

If $X$ and $Y$ are two normed $F$-spaces, and $T\in\operatorname{Hom}_F(X,Y)$, then we may define $\|T \|_v$ in the usual way
\begin{align*}
    \|T\|_v = \sup_{0\neq x\in X_v} \frac{\|Tx\|_v}{\|x\|_v}
\end{align*}
for all places $v$. As $\| T\|_v=1$ for almost all $v$, we may also define the adelic operator norm $\|T\| = \prod_v \|T\|_v$. For any finite set of places $S$, we define $\| T\|^S = \prod_{v\notin S}\|T\|_v$ and $\| T\|_S = \prod_{v\in S}\|T\|_v$.
If $S=\{v\}$ consists of a single place, we shall also write $\|T\|^S$ as $\|T\|^v$. 
As with vectors, these operator norms do not satisfy the triangle inequality, but they do satisfy $\|TU\|\leq\|T\|\|U\|$.

If we choose a basis $\{x_1,\dots,x_n\}$ for $X$, we can define a standard adelic norm in the following way. At infinite places, we define the norm by requiring the basis to be orthonormal, and at finite places, we require that the unit ball associated to the norm be the lattice spanned by the basis. We shall refer to these norms as the standard norm with respect to the basis. 

The following two lemmas help us to show that certain $F$-linear maps have nontrivial kernels. Lemma \ref{norm bound for vector in kernel} bounds the complexity of the kernel. Let $X$ and $Y$ be two  $F$-spaces, with bases  $\{x_1,\dots,x_n\}$ and $\{y_1,\dots,y_m\}$, and let $\|\cdot\|_X$ and $\|\cdot\|_Y$ be the corresponding standard norms. If $T\in\operatorname{Hom}_F(X,Y)$, we define $^t T\in\operatorname{Hom}_F(Y,X)$ the transpose of $T$ with repect to the chosen bases, and we have $\|T\|_v = \|^tT\|_v$ for all $v$.

\begin{lemma}\label{lemma 1.}
    Let $T\in\GL(X)$. Let $v$ be a place of $F$, and extend the norm $|\cdot|_v$ on $F_v$ to the algebraic closure $\overline{F_v}$. If $\lambda_v$ is an eignevalue of $T_v$, then
    \begin{align*}
        |\lambda_v|_v \geq \frac{1}{\| T\|^{n-1}\|T\|^v}.
    \end{align*}
\end{lemma}
\begin{proof}
    Suppose that $\lambda_{v,1} = \lambda_v$, $\lambda_{v,2},\dots,\lambda_{v,n}$ are the eigenvalues of $T_v$ repeated with multiplicity. For any place $w$ we have $| \det T|_w \leq \| T \|_w^n$, and so the product formula $\prod_w |\det T|_w = 1$ implies that $|\det T|_v\geq \left(\|T\|^v \right)^{-n}$. Since $|\lambda_{v,i}|\leq \|T\|_v$ for all $i$, we have
    \begin{align*}
        |\lambda_v| \|T \|_v^{n-1} \geq \prod_i |\lambda_{v,i}|_v = |\det T|_v\geq \left(\|T\|^v \right)^{-n},
    \end{align*}
    which completes the proof.
\end{proof}

\begin{lemma}\label{existance of kernel}
    Let $T\in\operatorname{Hom}_F(X,Y)$.   Let $v$ be a real place of $F$. If there exists $x_v\in X_v$ with
    \begin{align*}
        \frac{\| Tx_v\|_v}{\|x_v\|_v} <  \frac{1}{\|T \|^{2n-1}\|T\|^v},
    \end{align*}
    then $T$ has a nontrivial kernel.
\end{lemma}
\begin{proof}
    We assume that $T$ is injective.  For any $0\neq x\in X$, at the real place $v$, we have
    \begin{align*}
        \langle ^tTTx,x\rangle_{X,v} = \langle Tx,Tx\rangle_{Y,v}\neq 0,
    \end{align*}
    where $\langle\cdot,\cdot\rangle_{X,v}$ and $\langle\cdot,\cdot\rangle_{Y,v}$ are the inner products associated to $\|\cdot\|_{X,v}$ and $\|\cdot\|_{Y,v}$. This implies that $^tTT$ is also injective, so $^tTT\in\GL(X)$. We have
    \begin{align*}
        \|^tTTx_v\|_v \leq \frac{\| Tx_v\|_v}{\|x_v\|_v} \|^tT \|_v \|x_v\|_v <  \frac{\|T \|_v }{\|T \|^{2n-1}\|T\|^v} \|x_v\|_v= \frac{1 }{\|T \|^{2n-2}(\|T\|^v)^2}\|x_v\|_v.
    \end{align*}
    Because $^tTT$ is symmetric and $v$ is real, if $\lambda_v$ is the eigenvalue of $^tT_vT_v$ with minimal absolute value, then
    \begin{align*}
        |\lambda_v|_v \leq \frac{ \|^tTTx_v\|_v}{\|x_v\|_v}<\frac{1 }{\|T \|^{2n-2}(\|T\|^v)^2},
    \end{align*}
    which contradicts the result from  Lemma \ref{lemma 1.} applied to ${}^tTT$.
\end{proof}

\begin{lemma}\label{norm bound for vector in kernel}
    If $T\in\operatorname{Hom}_F(X,Y)$ is nonzero and not injective, then $\ker T$ can be spanned by vectors $x\in X$ with  $\|x\|\ll\|T \|^{n-1}$, where the implied constant depends only on $X$ and $Y$.
\end{lemma}
\begin{proof}
    With respect to the bases, we think of $T$ as an $m$-by-$n$ matrix. We may assume that the rows of $T$ are linearly independent, because removing a row that is a linear combination of other rows does not change the kernel, and decreases $\|T\|$, and so $m\leq n-1$. We can construct an element in $\ker T$ via Cramer's rule as follows. We augment $T$ to an $(n-1)$-by-$n$ matrix $\tilde{T}$, by adding in rows equal to arbitrary basis vectors. We then take a vector $v_{\tilde{T}}$ whose entries are the determinants of $(n-1)$-by-$(n-1)$ minors of $\tilde{T}$, times a suitable sign factor. The vectors $v_{\tilde{T}}$ span $\ker T$ as we vary over all possibilities for $\tilde{T}$. If the rows of $T$ are $v_1,\dots,v_m$, then we have $\|v_{\tilde{T}} \| \ll \prod_{i=1}^m\|v_i \| $ and $\|v_i\|\leq \|T\|$, which completes the proof.
\end{proof}

\subsection{Bounding $M(g,\delta, \fn)$}

As we shall use the results in the previous subsection, applied to the action of $\bG$ on $V$ via the representation $\rho$, we need to equip $V$ with an adelic norm. For infinite place $v$, we give $V_v$ an arbitrary norm; for finite place $v$, we give it the norm coming from $\cL_v$.
Since the group $H^\prime$ has two connected components, we may consider two subsets, for $g\in\Omega_{v_0}$  and $\delta>0$,
\begin{align*}
    M^+(g,\delta,\fn)=\left\{ \gamma\in\bG(F)\cap K(\fn) \,|\, d(g^{-1}\gamma g,e)\leq 1, d(g^{-1}\gamma g,H_0)\leq \delta\right\}
\end{align*}
and
\begin{align*}
    M^-(g,\delta,\fn)=\left\{ \gamma\in\bG(F)\cap K(\fn) \,|\, d(g^{-1}\gamma g,e)\leq 1, d(g^{-1}\gamma g,H^\prime\backslash H_0)\leq \delta\right\}.
\end{align*}
Consequently, $M(g,\delta,\fn) =  M^+(g,\delta,\fn) \cup  M^-(g,\delta,\fn)$.
For $X>0$, we define
\begin{align*}
    M^\pm(g,\delta,X) = \bigcup_{\substack{ \fn\text{ only divisible by primes in }\sP\\\operatorname{N}(\fn)<X }} M^\pm(g,\delta,\fn).
\end{align*}

\begin{lemma}\label{rep kernel lemma}
    Let $g\in\SL(2,\BC)$ so that $g\neq \pm I$. If $\rho(g)-I$ (resp. $\rho(g)+I$) has a nontrivial kernel, then the kernel $\ker(\rho(g)-I)$ (resp. $\ker(\rho(g)+I)$) must be two dimensional over $\BC$.
\end{lemma}
\begin{proof}
     It suffices to consider $V_{v_0}\otimes \BC$.
     Recall that this representation is given by $\rho(g) A = g A \bar{g}^{-1}$ for $g\in\SL(2,\BC)$  and $A\in M_2(\BC)$. Since $\rho(hgh^{-1})-I = \rho(h)(\rho(g)-I)\rho(h)^{-1}$, we can first reduce arbitrary $g\in \SL(2,\BC)$ to its Jordan form by taking the matrix conjugate in $\SL(2,\BC)$. Hence, we only need to deal with the following two cases.
    
    Suppose $g = \operatorname{diag}( \lambda,\lambda^{-1})$ with $\lambda\in \BC^*$ and $\lambda\neq \pm 1$. By acting on matrix units $E_{ij}$'s, it may be seen that $\rho(g)$ is diagonalizable with eigenvalues $|\lambda|^2, |\lambda|^{-2}, \lambda/\bar{\lambda}$ and $\bar{\lambda}/\lambda$. Then the result follows.

    The second case is that $g=\begin{pmatrix} 1&1\\&1\end{pmatrix}$, or $g=\begin{pmatrix} -1&1\\&-1\end{pmatrix}$. The lemma follows from a direct computation.
\end{proof}

\begin{lemma}\label{fixing a vector}
    There is a constant $C>0$ such that if $g\in\Omega_{v_0}$ and $X>0$, then there exist $\bv^\pm\in V$ with $\|\bv^\pm\|\ll X^3$ such that $\rho(\gamma)\bv = \pm\bv$ for all  $\gamma\in M^\pm(g,CX^{-8},X)$.
\end{lemma}
\begin{proof}
    Let $\delta = CX^{-8}$ for $C>0$ to be chosen later.
    We first show that for arbitrary three elements $\gamma_1,\gamma_2,\gamma_3\in M^+(g,\delta,X)$, 
    we have $\bigcap_{i=1}^3\ker(\rho(\gamma_i)-I)\neq 0$. We consider the operator $T:V\to V^{\oplus 3}$ given by $T = \oplus_i \alpha_i (\rho(\gamma_i)-I)$ for suitably chosen $\alpha_i\in F^\times$. For $v\in S$, we have $\| \rho(\gamma_i)-I\|_v\ll 1$ because $\gamma_i$'s lie in a fixed bounded set at places in $S$. If $\gamma_i \in M(g,\delta,\fn_i)$, an elementary computation gives $\| \rho(\gamma_i)\|_v = \operatorname{N} (\fn_{i,v})  $  for $v\notin S$, and hence $\| \rho(\gamma_i)-I\|_v \ll \operatorname{N} (\fn_{i,v})  $.
    We choose $\alpha_i$ such that $|\alpha_i|_v = \operatorname{N} (\fn_{i,v}^{-1})$ for $v\notin S$,  $|\alpha_i|_v \asymp 1$ for $v\in S$, $v\neq v_0$, and $|\alpha_i|_{v_0}\asymp \operatorname{N} (\fn_i)$. Such a choice can be made after possibly enlarging $S$ so that it generates the ideal class group of $F$. It follows that $\| \alpha_i (\rho(\gamma_i)-I)\|_v \leq 1$ for $v\notin S$, $\| \alpha_i (\rho(\gamma_i)-I)\|_v \ll 1$ for $v_0\neq v\in S$, and $\| \alpha_i (\rho(\gamma_i)-I)\|_{v_0} \ll\operatorname{N}(\fn_i) < X$. The same bounds therefore hold for $T$, so that $\|T \|\ll X$ and $\| T\|^{v_0}\ll 1$. The condition $d(g^{-1}\gamma_i g,H_{0})\leq \delta$ implies that $\|\rho(g^{-1}\gamma_ig)\bv_1-\bv_1 \|_{v_0} \ll \delta$, and as $g$ is bounded we also have $\|\rho(\gamma_i)\rho(g)\bv_1-\rho(g)\bv_1 \|_{v_0} \ll \delta$, so we also have $\|T\rho(g)\bv_1\|_{v_0}\ll \delta X$ and $\|\rho(g)\bv_1\|_{v_0}\asymp 1$. Therefore, by applying Lemma \ref{existance of kernel}, we see that $T$ has a nontrivial kernel unless
    \begin{align*}
        \delta X \gg \frac{\|T\rho(g)\bv_1\|_{v_0}}{\|\rho(g)\bv_1\|_{v_0}} \geq \frac{1}{\|T \|^{7}\|T\|^{v_0}} \gg X^{-7}.
    \end{align*}
    It follows that if we take $C$ small enough, then $T$ has a nontrivial kernel, which implies $\bigcap_{i=1}^3\ker(\rho(\gamma_i)-I)\neq 0$.
    
    By Lemma \ref{rep kernel lemma}, if $\pm I\neq \gamma\in\bG(F)$ is such that $\rho(\gamma)-I$ has a nontrivial kernel, then the kernel must be two dimensional. Moreover, by what we have proved above, if we take any three $\gamma_i\in M^+(g,\delta,X)\backslash\{\pm I\}$, then $\bigcap_{i}\ker(\rho(\gamma_i)-I)\neq 0$. This implies that
    \begin{align*}
        W  = \bigcap_{\gamma\in M^+(g,\delta,X)}\ker(\rho(\gamma)-I) \neq 0.
    \end{align*}
    If $\dim W=2$, then $W$ contains a vector with norm $\ll X^3$ by Lemma \ref{norm bound for vector in kernel}. If $\dim W =1$, then we may choose $\gamma_1,\gamma_2\in M(g,\delta,X)$ such that $W = \ker\left( \alpha_1(\rho(\gamma_1)-I)\oplus  \alpha_2(\rho(\gamma_2)-I)\right)$, where $\alpha_1,\alpha_2$ are chosen as above. 
    Lemma \ref{norm bound for vector in kernel} then gives $\bv\in W$ with $\| \bv\|\ll X^3$.
    
    The proof for $M^-(g,\delta,X)$ is the same except by taking $\rho(\gamma)+I$ instead of $\rho(\gamma)-I$.
\end{proof}

\begin{proof}[Proof of Proposition \ref{point counting prop}]
We let $C>0$ be as in Lemma \ref{fixing a vector}. We apply Lemma \ref{fixing a vector} to produce vectors $\bv^\pm\in V$ such that $\| \bv^\pm\|\ll X^3$ and $ \rho(\gamma)\bv^\pm=\pm\bv^\pm$ for all $\gamma\in M^\pm(g,CX^{-8},X)$. After scaling $\bv^\pm$, we assume that $\|\bv^\pm\|_v\leq 1$ for $v\neq v_0$ and that $\|\bv^\pm\|_{v_0}\ll X^3$. This implies that $\bv^\pm\in \cL_v=\{ (T,\iota(T))|T\in M_2(\cO_v) \}$ for $v\in\sP$. 

Let us construct the set $\sP_0$.
We have $\prod_v |\operatorname{N}_{E/F}(\operatorname{nrd}_B(\bv^\pm))|_v = 1$. We also have $|\operatorname{N}_{E/F}(\operatorname{nrd}_B(\bv^\pm))|_v\leq 1$ for $v\notin S$, $|\operatorname{N}_{E/F}(\operatorname{nrd}_B(\bv^\pm))|_v\ll 1$ for $v_0\neq v\in S$, and $|\operatorname{N}_{E/F}(\operatorname{nrd}_B(\bv^\pm))|_{v_0}\ll X^6$. 
It follows that if we let $\sP^\prime\subset\sP$ be the set of places at which $|\operatorname{N}_{E/F}(\operatorname{nrd}_B(\bv^\pm))|_v <1$, then we have $\sum_{v\in \sP^\prime}\log q_v\ll \log X$, and hence that $\#(\sP^\prime)\ll \log X$. We set $\sP_0 = \sP\backslash\sP^\prime$.
If $v\in\sP$, and if we know $|\operatorname{N}_{E/F}(\operatorname{nrd}_B(\bv^\pm))|_v=1$,  then $\bv^\pm\in\cL_v$ will imply $\bv^\pm = (T^\pm,\iota(T^\pm))$ for some $T^\pm \in \GL(2,\cO_v)$.
So the set $\sP_0$ is chosen to make sure that $\bv^\pm\in \GL(2,\cO_v)\times \GL(2,\cO_v)$ for $v\in \sP_0$. 

It remains to show that if $\fn$ is divisible only be primes in $\sP_0$ and satisfies $\operatorname{N}(\fn)<X$, $\fn\neq \cO$, then $M(g,\delta,\fn)= M^+(g,\delta,\fn) \cup  M^-(g,\delta,\fn)=\emptyset$.
Let $\gamma\in M^+(g,\delta,\fn)$, so that $\rho(\gamma)\bv^+=\bv^+$ by Lemma \ref{fixing a vector}. Let $v\in \sP_0$ dividing $\fn$. The condition $\gamma\in K(\fn)$ implies that $\gamma\in K_{v,1}tK_{v,1}\times K_{v,2} $ for some diagonal matrices $t$ so that $t\notin K_{v,1}$. It follows that $\overline{\iota(\gamma)}\in K_{v,1}\times K_{v,2}tK_{v,2} $. As $v\in\sP_0$, we have $\bv^+\in \GL(2,\cO_v)\times\GL(2,\cO_v)$. It follows that $\rho(\gamma)\bv^+ = \gamma\bv^+\overline{\iota(\gamma)} \in \GL(2,\cO_v)t\GL(2,\cO_v)\times \GL(2,\cO_v)t\GL(2,\cO_v)$, which contradicts that $\rho(\gamma)\bv^+=\bv^+$. The argument is the same if $\gamma\in M^-(g,\delta,\fn)$, which completes the proof.
\end{proof}

\section{Estimates of \texorpdfstring{$I(\lambda,g)$}{I(lambda,g)}}\label{integral}

This section aims to prove Proposition \ref{bound of I}. Recall that
\begin{align*}
     I(\lambda,g) = \iiiint_\BR b(x_1,t_1)\overline{b(x_2,t_2)} k_{\lambda}(a(-t_1)n(-x_1)gn(x_2)a(t_2))  e^{-t_1-t_2}   dx_1 dt_1 dx_2  dt_2.
\end{align*}
Applying the inverse Harish-Chandra transform \eqref{eq: HC inversion} to $k_\lambda$, we have
\begin{align*}
    I(\lambda,g) = \int_0^\infty h_\lambda(s) J(s,g) d\mu(s),
\end{align*}
where
\begin{align}
    J(s,g) &= \iiiint_\BR b(x_1,t_1)\overline{b(x_2,t_2)} \varphi_s(a(-t_1)n(-x_1)gn(x_2)a(t_2)) e^{-t_1-t_2}    dx_1 dt_1 dx_2  dt_2\notag\\
    &=\int_{\BR^4}\int_{K_0} b_0(x_1,t_1,x_2,t_2,u) \exp(isA(ua(-t_1)n(-x_1)g(n(x_2)a(t_2))  du   dx_1 dt_1 dx_2  dt_2.\label{eq:defn of J(s,g)}
\end{align}
The last identity above is obtained by applying \eqref{eq:HC K int formula} and writing $$b_0(x_1,t_1,x_2,t_2,u) = b(x_1,t_1)\overline{b(x_2,t_2)}\exp(A(ua(-t_1)n(-x_1)gn(x_2)a(t_2))-t_1-t_2) .$$
To bound $I(\lambda,g)$, it suffices to bound $J(s,g)$. Since $d\mu(s)/ds\asymp s^2$, we shall show the following proposition, which implies the desired bound on $I(\lambda,g)$.
\begin{proposition}\label{bound for J}
    If $s\geq1$ and $g\in G_0$ satisfies $d(g,e)\leq 1$, we have
    \begin{align*}
         J(s,g)\ll s^{-2}\left(1+s\,d(g,H^\prime)\right)^{-1}.
    \end{align*}
   Here the implied constant only depends on the support of $b$ and the $L^\infty$-norms of $b$ and finitely many of its derivatives.
\end{proposition}

We shall use the method of stationary phase to estimate the oscillatory integral $J(s,g)$ as $s\to \infty$.
    The following is a known generalization of
    the stationary phase approximation to the case of Morse–Bott functions.  See e.g., \cite[\S4]{deVerdiere}. We say a function $F\in C^\infty(\BR^d)$ is Morse-Bott if $F$ is a smooth function whose critical set is a closed submanifold and whose Hessian in the directions transverse to the critical set is non-degenerate.
    
    \begin{proposition}\label{stationary phase}
        Let $\alpha,F\in C^\infty(\BR^d)$ with $\alpha$ of compact support.
        Suppose that $F$ is Morse-Bott and that the set of critical points of $F$ contained in the support of $\alpha$ form a connected submanifold $W\subset\BR^d$ with  $\dim W=e$. Then the oscillatory integral
        \begin{align*}
            \int_{\BR^d} \alpha(x) e^{itF(x)} dx
        \end{align*}
        is asymptotic as $t\rightarrow\infty$ to 
        \begin{align*}
            \left(\frac{2\pi}{t}\right)^{(d-e)/2}e^{itF(W)-\frac{i\pi}{4}\sigma} \int_W \alpha(x) \left| \operatorname{det}_W F^{\prime\prime}(x)\right|^{-1/2} dx + O(t^{-(d-e)/2 -1}).
        \end{align*}
        Here $F(W)$ is the value of $F(x)$ at any point $x\in W$, and $\sigma$ (resp. $\det_WF^{\prime\prime}$) is the signature (resp. determinant) of the Hessian of $F$ in the directions transverse to $W$. The implicit constant depends on $F$, $\supp(\alpha)$, the $L^\infty$-norms of $\alpha$ and finitely many of its derivatives.
    \end{proposition}
\subsection{Outline of the section}\label{outline of the section} 
Note that, as in \eqref{eq:defn of J(s,g)}, $J(s,g)$ is an oscillatory integral in the variables $x_1,t_1,x_2,t_2,u$, with the phase function $A(ua(-t_1)n(-x_1)g(n(x_2)a(t_2)) $. In 
Section \ref{sec cal for A}, we build up some preliminary calculations for the Iwasawa projection $A$.

We shall study the critical points of the phase function in Section \ref{sec cri pt}.
After making an appropriate change of variables, we write the phase function as in \eqref{defn of phi}:
\begin{align*}
    - A\left( kn(x_1)a(t_1) \right) + A\left( kgn(x_2)a(t_2) \right)
\end{align*}
with variables $x_1,t_1,x_2,t_2\in\BR$ and $k\in K_0$.
In Proposition \ref{critical points} and Corollary \ref{corollary of critical point description}, we give a geometric description of the set of critical points. 
More precisely, if $g\in G_0\backslash H^\prime$ is fixed, then a critical point exists exactly when the distance between $\BH^2$ and $g\BH^2$ is positive in $\BH^3$.
Moreover, $(x_1^\prime,t_1^\prime,x_2^\prime,t_2^\prime,k^\prime)$ is a critical point exactly when the geodesic segment $v_0$ joining the pair $ (x_1^\prime,e^{t_1^\prime})$ and $g.(x_2^\prime,e^{t_2^\prime})$ realizes the distance between $\BH^2$ and $ g\BH^2 $, and $k^\prime v_0$ is a vertical geodesic. In particular, the set of all critical points will form two pairs of $\mathrm{U}(1)$-orbits in $\BR^4\times K_0$.

Now we suppose that $g\in G_0\backslash H^\prime$ and $(x_1^\prime,t_1^\prime,x_2^\prime,t_2^\prime,k^\prime)$ is a critical point. In Section \ref{subsection hessians}, we study the Hessian in the directions transversal to the critical set. We first choose an appropriate local chart for $K_0$ through the exponential map as in \eqref{phi tilde}. Let $h$ be the signed distance from $\BH^2$ to $g\BH^2$. By direct computations, Proposition \ref{Hessian} shows that the determinant of the Hessian in the directions transverse to the critical set is $\asymp (1-e^{2h})^2$. Therefore, if $|h|\geq \delta$ for some fixed $\delta>0$, then Proposition \ref{stationary phase} implies that $J(s,g)\ll_\delta s^{-3}$. However, this method fails when $h$ tends to 0.

We treat the situation where $h$ is close to 0 in Sections \ref{sec degenerate case} and \ref{sec degenerate int est}. We introduce a new phase function $\psi$ defined on $K_0\times G_0$ except for a measure zero subset. Lemma \ref{integral product lemma} explains how an oscillatory integral as in \eqref{eq:defn of J(s,g)} can be reduced to an oscillatory integral over $K_0$ with the phase function $\psi(\cdot,g)$. Furthermore, Lemma \ref{lem critical pt of psi} and Corollary \ref{Hessian of psi} suggest that the set of critical points of $\psi(\cdot,g)$ becomes degenerate when $h\to 0$ as well. More generally, we may assume that $d(g,H^\prime)$ tends to 0. We can write $g=g_0 \exp(Y)$ with $g_0\in H^\prime$ and $Y\in\fsl(2,\BR)^\perp$. Here $\fsl(2,\BR)^\perp$ is the subspace perpendicular to the Lie algebra of $H^\prime$ in $\fsl(2,\BC)$, which is defined in \eqref{eq: sl2 perp}, and $\|Y\|$ tends to 0. We blow up the origin in $\fsl(2,\BR)^\perp$ by introducing polar coordinates for $Y$ and we construct a Morse-Bott phase function from $\psi(\cdot,g_0\exp(Y))$. Therefore, we can apply Proposition \ref{stationary phase} to bound the oscillatory integral over $K_0$ with the phase $\psi(\cdot,g_0\exp(Y))$. This technique is explained in Proposition \ref{eliminating degeneracy} and Lemma \ref{lemma of estimate in Lie algebra} in detail. 

Finally, in Section \ref{sec bound for J}, we use the results proved in Sections \ref{sec cri pt}--\ref{sec degenerate int est} to complete the proof of Proposition \ref{bound for J}.

\subsection{Calculations for $A$}\label{sec cal for A}
We begin with some calculations for the Iwasawa projection $A$. The relation \eqref{Poincare} implies the following formulas on $A$.
\begin{lemma}\label{lemma A}
    Given $g\in \SL(2,\BC)$ with the Iwasawa decomposition $g=nak$ where $a=a(t_0)\in A$ and $k =\begin{pmatrix} \alpha&\beta\\-\bar{\beta}&\bar{\alpha}\end{pmatrix}\in K_0$, we have
    \begin{align}\label{iwasawa height}
        A(gn(z)a(t)) = t_0 + t - \log\left(|\alpha|^2+|\beta|^2(|z|^2+e^{2t}) - \alpha\bar{\beta}z-\bar{\alpha}\beta \bar{z}  \right).
    \end{align}
    By taking derivatives with respect to $x,t\in\BR$, we obtain
    \begin{align}\label{derivative of A about n}
        \frac{\partial}{\partial x} A(gn(x))|_{x=0} = \alpha\bar{\beta} + \bar{\alpha}\beta
    \end{align}
    and
    \begin{align}\label{derivative of A about a}
        \frac{\partial}{\partial t} A(ga(t))|_{t=0} = |\alpha|^2-|\beta|^2.
    \end{align}
\end{lemma}

For simplicity, we define $\Psi$ and $\Theta$ to be the functions on $K_0$ given by sending $k =\begin{pmatrix} \alpha&\beta\\-\bar{\beta}&\bar{\alpha}\end{pmatrix}\in K_0$ to $\alpha\bar{\beta} + \bar{\alpha}\beta$ and $|\alpha|^2-|\beta|^2$ respectively. Therefore, \eqref{derivative of A about n} and \eqref{derivative of A about a} can be read as
\begin{align*}
    \frac{\partial}{\partial x} A(gn(x))|_{x=0} = \Psi(\kappa(g))\quad\text{ and }\quad \frac{\partial}{\partial t} A(ga(t))|_{t=0} = \Theta(\kappa(g)).
\end{align*}
Recall that $\kappa:G_0\to K_0$ is the projection given by the Iwasawa decomposition $G_0=NAK_0$ and $l\subset \BH^3$ is the geodesic as in \eqref{defn of l}.

\begin{proposition}\label{crtical for x}
    Suppose $k\in K_0$. The following conditions are equivalent
    \begin{enumerate}
        \item $\frac{\partial}{\partial x}A(kn(x))|_{x=0} = \frac{\partial}{\partial t}A(ka(t))|_{t=0} =0$;
        \item $\Psi(k) = \Theta(k) = 0$;
        \item Either $k = \begin{pmatrix}e^{it}&0\\0&e^{-it}\end{pmatrix}\begin{pmatrix}1/\sqrt{2}& i/\sqrt{2}\\i/\sqrt{2}&1/\sqrt{2}\end{pmatrix}$, or $k = \begin{pmatrix}e^{it}&0\\0&e^{-it}\end{pmatrix}\begin{pmatrix}1/\sqrt{2}& -i/\sqrt{2}\\-i/\sqrt{2}&1/\sqrt{2}\end{pmatrix}$ where $t \in \BR/2\pi\BZ$;
        \item $l$ is perpendicular to $k\BH^2$ at $(0,1)$.
    \end{enumerate}
\end{proposition}
\begin{proof}
    The equivalence of $(a)$ and $(b)$ follows from Lemma \ref{lemma A}.
    The equivalence of $(b)$ and $(c)$ is obtained by solving $|\alpha|^2+|\beta|^2=1$, $|\alpha|^2-|\beta|^2=0$, and $\alpha\bar{\beta} +\bar{\alpha}\beta=0$. The equivalence of $(c)$ and $(d)$ can be seen from looking at the tangent space of $(0,1)$ in $\BH^2$.
    The matrices $\begin{pmatrix}1/\sqrt{2}& \pm i/\sqrt{2}\\\pm i/\sqrt{2}&1/\sqrt{2}\end{pmatrix}$ rotate the normal vector to the vertical directions and notice that $\begin{pmatrix}e^{it}&0\\0&e^{-it}\end{pmatrix}$ acts on $\BH^3$ by rotating around $l$ by angle $2t$.
\end{proof}
We will fix a basis for $\fk$
\begin{align*}
    X_1 = \begin{pmatrix}0&i\\i&0\end{pmatrix}, \, X_2 = \begin{pmatrix}0&-1\\1&0\end{pmatrix}, \, X_3 = \begin{pmatrix}i&0\\0&-i\end{pmatrix},
\end{align*}
and use these notations in the rest of this paper.
\begin{proposition}\label{critical for k}
    An element $g\in G_0$ lies in $AK_0$ if and only if $\frac{\partial}{\partial t} A(\exp(tX)g)|_{t=0} = 0$ for arbitrary $X\in\fk$.
\end{proposition}
\begin{proof}
    Without loss of generality, assume $g = a(t_0)\in A$. We only have to check $\frac{\partial}{\partial t} A(\exp(tX)g)|_{t=0} = 0$ holds for a basis of $\fk$.  By (\ref{iwasawa height}), we have $A(\exp(tX_1)a(t_0)) =A(\exp(tX_2)a(t_0)) = t_0 -\log(\cos^2t + e^{2t}\sin^2t)$ and $A(\exp(tX_3)a(t_0))=t_0 $, so $\frac{\partial}{\partial t} A(\exp(tX_i)g)|_{t=0} = 0$ with $i=1,2,3$.
    
    Conversely, we suppose $g\notin AK$ but $\frac{\partial}{\partial t} A(\exp(tX)g)|_{t=0} = 0$ for arbitrary $X\in\fk$. We can assume that $g = n(z_0)a(t_0)$ with $z_0\neq 0$. By (\ref{iwasawa height}), we have
    \begin{align*}
        \frac{\partial}{\partial t} A(\exp(tX_2)g)|_{t=0} &=  \frac{\partial}{\partial t}\left[ t_0  - \log\left(\cos^2(t)+\sin^2(t)(|z_0|^2+e^{2t_0}) + \cos (t)\sin (t)(z_0+\bar{z}_0) \right)\right]|_{t=0}\\
        &=z_0 + \bar{z}_0 = 0,
    \end{align*}
    implying $z_0 = iy_0\neq 0$ is purely imaginary. But
    \begin{align*}
        \frac{\partial}{\partial t} A(\exp(tX_1)g)|_{t=0} &=  \frac{\partial}{\partial t}\left[ t_0  - \log\left(\cos^2(t)+\sin^2(t)(y_0^2+e^{2t_0}) + 2\cos (t)\sin (t)y_0 \right)\right]|_{t=0}\\
        &=2y_0 = 0
    \end{align*}
    gives a contradiction.
\end{proof}

We show that the right translation by $G_0$ will induce an action of $G_0$ on $K_0$. For $g\in G_0$, let $\Phi_g: K_0\to K_0$ be the map sending $k$ to $\kappa(kg)$, i.e $kg\in NA\Phi_g(k)$. 
\begin{lemma}\label{change of variable}
    $g\mapsto\Phi_g$ is a smooth group action of $G_0$ on $K_0$ from right.
\end{lemma}
\begin{proof}
    The smoothness of the Iwasawa decomposition implies that $\Phi_g$ is smooth and depends smoothly on $g$. 
     We identify $K_0$ with the quotient $NA\backslash G_0$ via the Iwasawa decomposition. Then $\Phi_g$ is obtained by composing the diffeomorhpism $K_0\to NA\backslash G_0$, the right multiplication by $\cdot g:NA\backslash G_0\to NA\backslash G_0$ and the diffeomorphism $ NA\backslash G_0\to K_0$. It remains to show $\Phi_{gh} = \Phi_h\circ\Phi_g$ for $g,h\in G_0$. By definition
     \begin{align*}
         kgh\in NA\Phi_{gh}(k)\\
         kg\in NA\Phi_g(k),
     \end{align*}
     which implies
     \begin{align*}
         NA\Phi_g(k) h = NA\Phi_{gh}(k),\\
         \Phi_g(k) h \in NA\Phi_{gh}(k),
     \end{align*}
     that is $\Phi_h\left(  \Phi_g(k) \right) = \Phi_{gh}(k)$.
\end{proof}

The following lemma will help us to rewrite the phase function of $J(s,g)$  so that it is easier to find its critical points.
\begin{lemma}\label{spliting A}
    Let $y,z\in G_0$ and let $k\in K_0$. Then we have
    \begin{align*}
        A(ky^{-1}z) = A\left(\Phi_{y^{-1}}(k)z\right)- A\left(\Phi_{y^{-1}}(k)y\right).
    \end{align*}
\end{lemma}
\begin{proof}
    Let $ky^{-1} = na\Phi_{y^{-1}}(k)$. We have
    \begin{align*}
        A(ky^{-1}z) &= A\left( na\Phi_{y^{-1}}(k)z\right)\\
        &=A(na) + A\left( \Phi_{y^{-1}}(k)z\right)\\
        &=A\left( \Phi_{y^{-1}}(k)z\right) - A\left((na)^{-1}k\right)\\
        &=A\left( \Phi_{y^{-1}}(k)z\right) - A\left(\Phi_{y^{-1}}(k)y\right).
    \end{align*}
\end{proof}

\subsection{Critical Points}\label{sec cri pt}
Applying Lemma \ref{change of variable} and Lemma \ref{spliting A}, and writing $k = \Phi_{a(-t_1)n(-x_1)}(u)$, we can rewrite the integral $J(s,g)$ as 
\begin{align*}
&J(s,g)= \int_{\BR^4}\int_{K_0} b_0(x_1,t_1,x_2,t_2,u) \exp(isA(ua(-t_1)n(-x_1)gn(x_2)a(t_2)))  du   dx_1 dt_1 dx_2  dt_2\\
    =&\int_{\BR^4}\int_{K_0} b_0(x_1,t_1,x_2,t_2,u) \exp\left(is\left(-A(kn(x_1)a(t_1))+A(kgn(x_2)a(t_2))\right)\right)  du   dx_1 dt_1 dx_2 dt_2\\
    =&\int_{\BR^4}\int_{K_0} b_0(x_1,t_1,x_2,t_2,u)\left|\det\mathbf{J}\Phi_{n(x_1)a(t_1)}(k)\right|\exp\left(is\left(-A(kn(x_1)a(t_1))+A(kgn(x_2)a(t_2))\right)\right) dk   dx_1 dt_1 dx_2 dt_2. 
\end{align*}
To simplify the notation, we let $b(x_1,t_1,x_2,t_2,k)=b_0(x_1,t_1,x_2,t_2,\Phi_{n(x_1)a(t_1)}(k))\left|\det\mathbf{J}\Phi_{n(x_1)a(t_1)}(k)\right|$, and define $\phi$ as a function on $\BR^4 \times K_0 \times G_0$ by
\begin{align}\label{defn of phi}
    \phi(x_1,t_1,x_2,t_2,k,g) = - A\left( kn(x_1)a(t_1) \right) + A\left( kgn(x_2)a(t_2) \right).
\end{align}
We will omit the variable $g$ when it is fixed. Hence,
\begin{align*}
    J(s,g) = \int_{\BR^4}\int_{K_0} b(x_1,t_1,x_2,t_2,k)\exp(is\phi(x_1,t_1,x_2,t_2,k,g))  dk   dx_1 dt_1 dx_2 dt_2.
\end{align*}
Combining Proposition \ref{crtical for x} and Proposition \ref{critical for k}, we can describe the set of critical points of $\phi$.
\begin{proposition}\label{critical points}
    Suppose $g$ is fixed.
    The phase function $\phi$ has a critical point at $(x_1,t_1,x_2,t_2,k)$ exactly when $k.(x_1 ,e^{t_1})$ and $kg.(x_2,e^{t_2})$ lie on the same vertical geodesic $v$, and $v$ is perpendicular to $k\BH^2$ at $k.(x_1 ,e^{t_1})$ and to $kg\BH^2$ at $kg.(x_2,e^{t_2})$.
\end{proposition}

\begin{proof}
    Suppose that $(x_1^\prime,t_1^\prime,x_2^\prime,t_2^\prime,k^\prime)\in\BR^4\times K_0$ is a critical point of $\phi$. We define $n_1^\prime\in N,a_1^\prime\in A$ and $u_1^\prime \in K_0$ by
    \begin{align*}
        k^\prime n(x_1^\prime)a(t_1^\prime) = n_1^\prime a_1^\prime u_1^\prime .
    \end{align*}
    It may be seen that $v_1:=n_1^\prime l$ is the vertical geodesic through $k^\prime .(x_1^\prime,e^{t_1^\prime })$.  Recall that here $l$ is the vertical geodesic as in \eqref{defn of l}.   By Proposition \ref{crtical for x},
    \begin{align*}
        \frac{\partial\phi}{\partial x_1}(x_1^\prime,t_1^\prime,x_2^\prime,t_2^\prime,k^\prime)= \frac{\partial\phi}{\partial t_1}(x_1^\prime,t_1^\prime,x_2^\prime,t_2^\prime,k^\prime)=0
    \end{align*}
    implies that $l$ is perpendicular to $u_1^\prime \BH^2$ at $(0,1)$, and so $v_1 = n_1^\prime a_1^\prime  l$ is perpendicular to  $n_1^\prime a_1^\prime u_1^\prime\BH^2 = k^\prime\BH^2$ at
    $n_1^\prime a_1^\prime u_1^\prime .(0,1) =k^\prime .(x_1^\prime,e^{t_1^\prime })$.
    The calculation of $\partial/\partial x_2, \partial/\partial t_2$ is identical. If we define 
    $n_2^\prime\in N,a_2^\prime\in A$ and $u_2^\prime \in K_0$ by
    \begin{align*}
        k^\prime g n(x_2^\prime)a(t_2^\prime) = n_2^\prime a_2^\prime u_2^\prime ,
    \end{align*}
    then $v_2 := n_2^\prime l$ is the vertical geodesic perpendicular to $k^\prime g\BH^2$  at $k^\prime g.(x_2^\prime,e^{t_2^\prime })$.
    It remains to show $v_1 = v_2$.
    We define $u = \Phi_{n(x_1^\prime)a(t_1^\prime)}:K_0\to K_0$ and then $u_1^\prime = u(k^\prime)$.
    Lemma \ref{spliting A} gives
    \begin{align*}
        - A\left( kn(x_1^\prime)a(t_1^\prime) \right) + A\left( kgn(x_2^\prime)a(t_2^\prime) \right) = A\left( u(k)a(-t_1^\prime)n(-x_1^\prime)gn(x_2^\prime)a(t_2^\prime)\right).
    \end{align*}
    By Lemma \ref{change of variable}, $u(\cdot)$ is a diffeomorphism, so $k^\prime$ is a critical point exactly when $u_1^\prime$ is a critical point of the function $A\left( ua(-t_1^\prime)n(-x_1^\prime)gn(x_2^\prime)a(t_2^\prime)\right)$ in the variable $u$. This holds if and only if for any $X\in\fk$
    \begin{align*}
        \frac{\partial}{\partial t}A\left( \exp(tX)u^\prime a(-t_1^\prime)n(-x_1^\prime)gn(x_2^\prime)a(t_2^\prime)\right)|_{t=0} = 0.
    \end{align*}
    By Proposition \ref{critical for k}, this is equivalent to 
    \begin{align}\label{eq 5.5}
        u_1^\prime a(-t_1^\prime)n(-x_1^\prime)gn(x_2^\prime)a(t_2^\prime)\in AK_0; 
    \end{align}
    that is the point 
    \begin{align*}
        (u_1^\prime a(-t_1^\prime)n(-x_1^\prime)g).(x_2^\prime,e^{t_2^\prime}) = \left( u_1^\prime a(-t_1^\prime)n(-x_1^\prime)gn(x_2^\prime)a(t_2^\prime)\right). (0,1)
    \end{align*}
    lies on the vertical geodesic $l$. Since
    \begin{align*}
        u_1^\prime a(-t_1^\prime)n(-x_1^\prime) = a_1^{\prime-1}n_1^{\prime-1} k^\prime,
    \end{align*}
    $u_1^\prime$ being a critical point is equivalent to that
    \begin{align*}
        (a_1^{\prime-1}n_1^{\prime-1} k^\prime g).(x_2^\prime,e^{t_2^\prime})\in l,
    \end{align*}
    which is equivalent to that
    \begin{align*}
        k^\prime g.(x_2^\prime,e^{t_2^\prime}) \in n_1^\prime l =v_1.
    \end{align*}
    Hence, $k^\prime g.(x_2^\prime,e^{t_2^\prime})$ lies on both $v_1$ and $v_2$. We conclude $v_1 = v_2$ because they are both vertical.

     Conversely, if there is a vertical geodesic $v$ perpendicular to $k^\prime\BH^2$ at $k^\prime.(x_1^\prime,e^{t_1^\prime})$  and to $k^\prime g\BH^2$ at $k^\prime g.(x^\prime_2,e^{t^\prime_2})$, then we may conclude that
     \begin{align*}
         \frac{\partial\phi}{\partial x_i}(x_1^\prime,t_1^\prime,x_2^\prime,t_2^\prime,k^\prime) = \frac{\partial\phi}{\partial t_i}(x_1^\prime,t_1^\prime,x_2^\prime,t_2^\prime,k^\prime) = 0\quad\text{ for }i=1,2,
     \end{align*}
    by Proposition \ref{crtical for x}. We have already shown that points $k^\prime .(x_1^\prime,e^{t_1^\prime})$ and $k^\prime g.(x_2^\prime,e^{t_2^\prime})$ lie on the same vertical geodesic if and only if $u_1^\prime $ is critical. Hence, we conclude that $k^\prime $ is also critical.
\end{proof}

\begin{corollary}\label{corollary of critical point description}
Suppose $g\in G_0\backslash H^\prime$. Then
\begin{enumerate}
    \item The set of critical points of $\phi$ is nonempty if and only if the distance between $\BH^2$ and $ g\BH^2 $ is nonzero. 
    In this case, $(x_1^\prime,t_1^\prime,x_2^\prime,t_2^\prime,k^\prime)$ is a critical point exactly when the geodesic segment $v_0$ joining the pair $ (x_1^\prime,e^{t_1^\prime})$ and $g.(x_2^\prime,e^{t_2^\prime})$ realizes the distance between $\BH^2$ and $ g\BH^2 $, and $k^\prime v_0$ is vertical. 
    \item Moreover, if $(x_1^\prime,t_1^\prime,x_2^\prime,t_2^\prime,k^\prime)$ is a critical point of $\phi$, then all critical points of $\phi$ are of the form
    \begin{align*}
        \left\{(x_1^\prime,t_1^\prime,x_2^\prime,t_2^\prime,k)\in \BR^4\times K_0 \,|\, k\in \operatorname{U}(1)k^\prime \cup \operatorname{U}(1) \Phi_{n(x_1^\prime)a(t_1^\prime)}^{-1}(w_0 u^\prime) \right\}
    \end{align*}
    which form two pairs of $\operatorname{U}(1)$-orbits. Here $w_0 = \begin{pmatrix}
            &1\\-1&
        \end{pmatrix}$ is the Weyl element, and $u^\prime = \Phi_{n(x_1^\prime)a(t_1^\prime)}(k^\prime)$.
\end{enumerate}

\end{corollary}

Recall that the group $\mathrm{U}(1)$ is defined by \eqref{defn of U(1)}.

\begin{proof}

    By Proposition \ref{critical points}, if $(x_1^\prime,t_1^\prime,x_2^\prime,t_2^\prime,k^\prime)$ is a critical point of $\phi$, then there exists a vertical geodesic $v$ perpendicular to $k^\prime\BH^2$ at $k^\prime.(x_1^\prime ,e^{t_1^\prime})$ and to $k^\prime g\BH^2$ at $k^\prime g.(x_2^\prime,e^{t_2^\prime})$. Hence, the geodesic $k^{\prime -1}v$ is perpendicular to $\BH^2$ at $(x_1^\prime ,e^{t_1^\prime})$ and to $g\BH^2$ at $g.(x_2^\prime ,e^{t_2^\prime})$. Since $g\notin H^\prime$, the distance between $\BH^2$ and $g\BH^2$ must be positive and is realized by the segment of $k^{\prime -1}v$ joining $(x_1^\prime ,e^{t_1^\prime})$ and $g.(x_2^\prime ,e^{t_2^\prime})$.
    Conversely, we assume that the distance between $\BH^2$ and $ g\BH^2 $ is positive and the geodesic segment $v_0$ joining $ (x_1^\prime,e^{t_1^\prime})$ and $g.(x_2^\prime,e^{t_2^\prime})$ realizes the distance. If $k^\prime v_0$ is vertical, then the vertical geodesic $k^\prime v_0$ is perpendicular to $k^\prime\BH^2$ and $ k^\prime g\BH^2 $ at $ k^\prime.(x_1^\prime,e^{t_1^\prime})$ and $k^\prime g.(x_2^\prime,e^{t_2^\prime})$. We conclude that $(x_1^\prime,t_1^\prime,x_2^\prime,t_2^\prime,k^\prime)$ is a critical point of $\phi$ by Proposition \ref{critical points}, which completes the proof for part (a). 

    Let us prove the part (b).  Suppose $(x_1^\prime,t_1^\prime,x_2^\prime,t_2^\prime,k^\prime)$ is a given critical point of $\phi$, and $(x_1,t_1,x_2,t_2,k)$ is another one. 
    From part (a), we know that $(x_1,t_1,x_2,t_2) = (x_1^\prime,t_1^\prime,x_2^\prime,t_2^\prime)$ and $k^\prime v_0,kv_0$ are vertical. Hence, the vertical geodesic $l$ is perpendicular to both $ \Phi_{n(x_1^\prime)a(t_1^\prime)}(k^\prime) \BH^2$ and $ \Phi_{n(x_1^\prime)a(t_1^\prime)}(k) \BH^2$. From the equivalence of (c) and (d) in Proposition \ref{crtical for x}, we have
    \begin{align}\label{critical set for k using orbit U(1)}
        \Phi_{n(x_1^\prime)a(t_1^\prime)}(k) \in \operatorname{U}(1)u^\prime \cup \operatorname{U}(1) (w_0 u^\prime).
    \end{align}
    We can conclude that
    \begin{align*}
        k\in \operatorname{U}(1)k^\prime \cup \operatorname{U}(1) \Phi_{n(x_1^\prime)a(t_1^\prime)}^{-1}(w_0 u^\prime) ,
    \end{align*}
    by acting $\Phi_{n(x_1^\prime)a(t_1^\prime)}^{-1}$ on both sides of (\ref{critical set for k using orbit U(1)}), and applying the property that the diffeomorphism $\Phi_{n(x_1^\prime)a(t_1^\prime)}$ preserve the subgroup $\operatorname{U}(1)$.
\end{proof}

\subsection{The Hessians}\label{subsection hessians}

In this subsection, we assume $g\in G_0\backslash H^\prime$ satisfying that the distance between $\BH^2$ and $g\BH^2$ is positive.
Therefore, Corollary \ref{corollary of critical point description} implies that the set of critical points is of dimension one and $\operatorname{U}(1)$-invariant.
Since the set of critical points is of positive dimension, we shall care about the Hessian in the directions transversal to the critical set. Later, we will make this more precise using the local coordinates given by the exponential map to $K_0$.

Let $(x_1^\prime,t_1^\prime,x_2^\prime,t_2^\prime,k^\prime)\in\BR^4\times K_0$ be a fixed critical point of $\phi$, and define maps $u_1 = \Phi_{n(x_1^\prime)a(t_1^\prime)}$ and $u_2=\Phi_{gn(x_2^\prime)a(t_2^\prime)}$ from $K_0$ to $K_0$, that is
\begin{align}\label{u1 u2}
    kn(x_1^\prime)a(t_1^\prime) \in NA u_1(k),\quad\text{ and } \quad kgn(x_2^\prime)a(t_2^\prime) \in NA u_2(k).
\end{align}
We let $u_1^\prime=u_1(k^\prime)$ and $u_2^\prime = u_2(k^\prime)$. Lemma \ref{crtical for x} $(b)$ and $(c)$ implies that $\Psi(u_1^\prime)= \Psi(u_2^\prime)=0$, $\Theta(u_1^\prime) = \Theta(u_2^\prime)=0$, and they must be of the form
\begin{align}\label{definition of ujprime}
    u_j^\prime = \begin{pmatrix}e^{i\theta_j}&0\\0&e^{-i\theta_j}\end{pmatrix}\begin{pmatrix}1/\sqrt{2}& \epsilon_j i/\sqrt{2}\\ \epsilon_ji/\sqrt{2}&1/\sqrt{2}\end{pmatrix},
\end{align}
where $j=1,2$, $\theta_j\in\BR/2\pi\BZ$ and $\epsilon_j\in\{\pm 1\}$.
From \eqref{eq 5.5} in the proof of Proposition \ref{critical points}, we have seen that $$u_1^\prime a(-t_1^\prime)n(-x_1^\prime)gn(x_2^\prime)a(t_2^\prime)\in AK_0.$$
We let $h\in\BR$ be the number that
\begin{align}\label{definition of h}
    u_1^\prime a(-t_1^\prime)n(-x_1^\prime)gn(x_2^\prime)a(t_2^\prime)\in a(h)K_0.
\end{align}
It may be seen that $h$ is the signed distance from $k^\prime n(x_1^\prime)a(t_1^\prime)$ to $k^\prime gn(x_2^\prime)a(t_2^\prime)$ along the geodesic $v$.

To bound the oscillatory integral
\begin{align*}
    J(s,g) = \int_{\BR^4}\int_{K_0} b(x_1,t_1,x_2,t_2,k)\exp(is\phi(x_1,t_1,x_2,t_2,k,g))  dk   dx_1 dt_1 dx_2 dt_2,
\end{align*}
it suffices to bound the integral localized near the critical points. We define
\begin{align}\label{J tilde}
    \begin{split}
        &J(s,g;x_1^\prime,t_1^\prime,x_2^\prime,t_2^\prime,k^\prime)\\=& \int_{\BR^7} B(x_1,t_1,x_2,t_2,s_1,s_2,s_3)\exp(is\tilde{\phi}(x_1,t_1,x_2,t_2,s_1,s_2,s_3))  ds_1 ds_2 ds_3   dx_1 dt_1 dx_2 dt_2,
    \end{split}
\end{align}
where
\begin{align}\label{phi tilde}
\begin{split}
     &\tilde{\phi}(x_1,t_1,x_2,t_2,s_1,s_2,s_3)\\
    =&- A\left( u_1^{-1}(\exp(s_1X_1+s_2X_2+s_3X_3) u_1^\prime)n(x_1)a(t_1) \right)\\
    &+ A\left( u_1^{-1}(\exp(s_1X_1+s_2X_2+s_3X_3) u_1^\prime)gn(x_2)a(t_2) \right)
\end{split}
\end{align}
has a critical point at  $(x_1^\prime,t_1^\prime,x_2^\prime,t_2^\prime,0,0,0)$ and $B(x_1,t_1,x_2,t_2,s_1,s_2,s_3)$ is a smooth function on $\BR^7$ with a compact support $\supp B =\overline{S_1} \times \overline{S_2} \subset \BR^4\times\BR^3$. Here $S_1$ and $S_2$ are both bounded open, $\overline{S_1}$ and $ \overline{S_2} $ are their closures.
By the proof of Proposition \ref{critical points} or $(c)$ in Lemma \ref{crtical for x}, we know that $(x_1^\prime,t_1^\prime,x_2^\prime,t_2^\prime,0,0,t)$ is still a critical point of $\tilde{\phi}$. We may assume that $S_1$ and $S_2$ are small enough so that they satisfy the following two properties:
\begin{itemize}
    \item (A1) $(x_1^\prime,t_1^\prime,x_2^\prime,t_2^\prime,0,0,s_3)$ with $(0,0,s_3)\in S_2$ are the only critical points of $\tilde{\phi}$ in $S_1\times S_2$.
\end{itemize}
 We can identify $S_2$ with a neighborhood of $0$ in $\fk$ by sending $(s_1,s_2,s_3)$ to $s_1 X_1 + s_2 X_2 + s_3 X_3$. Then
\begin{itemize}
    \item (A2) The exponential map restricted to $S_2$ is a diffeomorphism onto its image in $K_0$.
\end{itemize}

We can decompose $J(s,g)$ into a finite sum of $ J(s,g;x_1^\prime,t_1^\prime,x_2^\prime,t_2^\prime,k^\prime)$'s by change of variables via the exponential map and the diffeomorphism $u_1$. Therefore, we only need to bound  $ J(s,g;x_1^\prime,t_1^\prime,x_2^\prime,t_2^\prime,k^\prime)$. The goal of this section is to compute the Hessian of $\tilde{\phi}$ at the critical point $(x_1^\prime,t_1^\prime,x_2^\prime,t_2^\prime,0,0,0)$.

When we calculate $\partial^2 \tilde{\phi}/\partial s_i\partial x_2$ and $\partial^2 \tilde{\phi}/\partial s_i\partial t_2$, the local coordinate for the $k$-variable is written under $u_1$, but it is easier to do so under $u_2$ when differentiating $x_2,t_2$, so we need to explore the relation between the maps $u_1$ and $u_2$ first.
Since $u_2\circ u_1^{-1}$ is a diffeomorphism, it is a local diffeomorphism sending $u_1^\prime$ to $u_2^\prime$. More precisely, we suppose $U_1$ is an open neighborhood of $0$ in $\fk$ so that the restriction of $\exp$ to $U_1$ is a diffeomorphism from $U_1$ to its image in $K_0$. Let $V_1 =\exp(U_1) u_1^\prime$ and $V_2 = u_2\circ u_1^{-1}(V_1)$. By assuming that $U_1$ is small enough, we may assume that $V_2$ is also diffeomorphic to some open neighborhood $U_2$ of $0$ in $\fk$ under $\exp(\cdot)u_2^\prime$. We define a map $\alpha_{12}:U_1 \to U_2$ by lifting the diffeomorphism $u_2\circ u_1^{-1}:V_1\to V_2$, that is, $\alpha_{12}$ fits into the commutative diagram
\begin{align}\label{commutative diagram}
\begin{CD}
U_1 @>\alpha_{12}>> U_2\\
@V\exp(\cdot)u_1^\prime VV @VV\exp(\cdot)u_2^\prime V\\
V_1 @>u_2\circ u_1^{-1}>> V_2.
\end{CD}
\end{align}
We identify $U_1$ and $U_2$ as open subsets of $\BR^3$ by fixing the basis $X_1,X_2,X_3\in\fk$, and will still use $s_1,s_2,s_3$ as coordinates of $U_1$ and will use $r_1,r_2,r_3$ as coordinates of $U_2$. More precisely, $(s_1,s_2,s_3)\in U_1$ corresponds to the vector $s_1X_1 + s_2 X_2 + s_3 X_3 \in \fk$, and $(r_1,r_2,r_3)\in U_1$ corresponds to the vector $r_1X_1 + r_2 X_2 + r_3 X_3 \in \fk$. So to allow the above local diffeomorphism $\alpha_{12}$, we may put the last assumption on $S_1\times S_2$:
\begin{itemize}
    \item (A3) $S_2\subset U_1$.
\end{itemize}

\begin{lemma}\label{Iwasawa k}
    For $g= \begin{pmatrix}a &b \\c &d \end{pmatrix}\in G_0$, its Iwasawa projection to $K_0$ is given by the formula
    \begin{align*}
        \kappa(g) = \frac{1}{\sqrt{|c|^2+|d|^2}}\begin{pmatrix}{\bar{d}} &-{\bar{c}} \\{c}&{d} \end{pmatrix},
    \end{align*}
    that is, $g\in NA\kappa(g)$.
\end{lemma}
\begin{proof}
    Directly follows from computing the action of $g$ on $\BH^3$.
\end{proof}

\begin{proposition}\label{jacobian u1 u2}
    For $X\in U_1$, we have
    \begin{align}\label{exp of alpha12}
        \exp(\alpha_{12}(X)) = \kappa (\exp(X)a(h)),
    \end{align}
    where $h$ is as in (\ref{definition of h}), and
    the Jacobian of $\alpha_{12}$ at $0$ is
    \begin{align*}
        \bJ\alpha_{12}(0,0,0) =\frac{\partial(r_1,r_2,r_3)}{\partial(s_1,s_2,s_3)}(0,0,0) = \begin{pmatrix}e^h& & \\&e^h&\\&&1 \end{pmatrix}.
    \end{align*}
\end{proposition}

\begin{proof}
    We recall that from (\ref{definition of h}) we have
    \begin{align*}
        u_1^\prime a(-t_1^\prime)n(-x_1^\prime)gn(x_2^\prime)a(t_2^\prime)= a(h)k_0
    \end{align*}
    for some $k_0\in K$, and so
    \begin{align*}
        kgn(x_2^\prime)a(t_2^\prime)=kn(x_1^\prime)a(t_1^\prime)u_1^{\prime-1}a(h)k_0,
    \end{align*}
    for all $k\in K$. Substituting both parts of (\ref{u1 u2}) into this gives
    \begin{align*}
        NAu_1(k)u_1^{\prime-1}a(h)k_0&= NAu_2(k) ,
    \end{align*}
    so that
    \begin{align*}
            u_1(k)u_1^{\prime-1}a(h)&\in NAu_2(k)k_0^{-1}.
        \end{align*}
        By setting $k=k^\prime$ we see that $k_0 = u_2(k^\prime)=u_2^\prime$, so
        \begin{align}\label{610}
            u_1(k)u_1^{\prime-1}a(h)&\in NAu_2(k)u_2^{\prime-1}.
        \end{align}
        Then, for $X\in U_1$, if we let $k = u_1^{-1}\left(\exp(X) u_1^\prime  \right) = u_2^{-1}(\exp(\alpha_{12}(X))u_2^\prime) $, we have
        \begin{align*}
            \exp(X)a(h) \in NA\exp(\alpha_{12}(X)),
        \end{align*}
        so that
        \begin{align*}
            \exp(\alpha_{12}(X)) = \kappa\left(\exp(X)a(h) \right),
        \end{align*}
        which proves (\ref{exp of alpha12}).
        
        By applying Lemma \ref{Iwasawa k} and (\ref{exp of alpha12}), we have, for a real number $t$ near $0$,
        \begin{align*}
            \exp(\alpha_{12}(tX_1)) &=\kappa\left( \exp(tX_1)a(h) \right)\\
            &=\kappa\left( \begin{pmatrix}e^{h/2}\cos t &ie^{-h/2}\sin t\\ie^{h/2}\sin t&e^{-h/2}\cos t\end{pmatrix} \right)\\
            &=\frac{1}{\sqrt{e^h\sin^2t+e^{-h}\cos^2t}}\begin{pmatrix}e^{-h/2}\cos t &ie^{h/2}\sin t\\ie^{h/2}\sin t&e^{-h/2}\cos t\end{pmatrix},
        \end{align*}
        so this gives
        \begin{align*}
            &\frac{\partial}{\partial t}\exp(\alpha_{12}(tX_1)) |_{t=0}\\
            =&\left.\frac{\partial}{\partial t}\left( \frac{1}{\sqrt{e^h\sin^2t+e^{-h}\cos^2t}}\begin{pmatrix}e^{-h/2}\cos t &ie^{h/2}\sin t\\ie^{h/2}\sin t&e^{-h/2}\cos t\end{pmatrix} \right)\right|_{t=0}\\
            =& e^h \begin{pmatrix}0&i\\i&0\end{pmatrix} \\
            =& e^h X_1.
        \end{align*}
        Since
        \begin{align*}
            \frac{\partial}{\partial t}\exp(\alpha_{12}(tX_1)) |_{t=0}=\exp(\alpha_{12}(0))  \frac{\partial\alpha_{12}}{\partial s_1}(0,0,0)=\frac{\partial\alpha_{12}}{\partial s_1}(0,0,0),
        \end{align*}
        we have
        \begin{align*}
            \frac{\partial\alpha_{12}}{\partial s_1}(0,0,0) = e^h X_1.
        \end{align*}
        Similarly, we have
        \begin{align*}
            \exp(\alpha_{12}(tX_2)) &=\frac{1}{\sqrt{e^h\sin^2t+e^{-h}\cos^2t}}\begin{pmatrix}e^{-h/2}\cos t &-e^{h/2}\sin t\\e^{h/2}\sin t&e^{-h/2}\cos t\end{pmatrix},
        \end{align*}
        and
        \begin{align*}
            \exp(\alpha_{12}(tX_3)) &=\begin{pmatrix}e^{it} &0\\0&e^{-it}\end{pmatrix}.
        \end{align*}
        We conclude that
        \begin{align*}
            \frac{\partial\alpha_{12}}{\partial s_2}(0,0,0)= e^h X_2,\quad\text{ and }\quad\frac{\partial\alpha_{12}}{\partial s_3}(0,0,0)=  X_3.
        \end{align*}
    \end{proof}
    
    \begin{lemma}\label{gradient of psi theta}
        Given $k=\begin{pmatrix}\alpha&{\beta}\\-\bar{\beta}&\bar{\alpha}\end{pmatrix}\in K_0$, the gradients of $\Psi$ and $\Theta$ at $k$ are
        \begin{align*}
            &\nabla_{K_0}\Psi(k):=\left. \left(   \frac{\partial}{\partial t}\Psi(\exp(tX_1)k) ,\frac{\partial}{\partial t}\Psi(\exp(tX_2)k) ,\frac{\partial}{\partial t}\Psi(\exp(tX_3)k)    \right)\right|_{t=0}\\
            =&\left(-i(\alpha^2-\bar{\alpha}^2-\beta^2+\bar{\beta}^2), -(\alpha^2+\bar{\alpha}^2-\beta^2-\bar{\beta}^2),0 \right),
        \end{align*}
        and
        \begin{align*}
            \nabla_{K_0}\Theta(k)=\left( 2i(\alpha\beta-\bar{\alpha}\bar{\beta}), 2(\alpha\beta+\bar{\alpha}\bar{\beta}), 0\right).
        \end{align*}
        In particular, if $k=  \begin{pmatrix}e^{i\theta}&0\\0&e^{-i\theta}\end{pmatrix}\begin{pmatrix}1/\sqrt{2}& \epsilon i/\sqrt{2}\\ \epsilon i/\sqrt{2}&1/\sqrt{2}\end{pmatrix}$ with $\theta\in\BR/2\pi\BZ$ and $\epsilon\in\{\pm1\}$, then
        \begin{align*}
            &\nabla\Psi(k)
            =\left(2\sin(2\theta), -2\cos(2\theta),0 \right),
        \end{align*}
        and
        \begin{align*}
            \nabla\Theta(k)=\left( -2\epsilon \cos(2\theta), -2\epsilon \sin(2\theta), 0\right).
        \end{align*}
    \end{lemma}
    \begin{proof}
        Since
        \begin{align*}
            \exp(tX_1)k =  \begin{pmatrix}\alpha \cos t-i\bar{\beta}\sin t& {\beta}\cos t +i\bar{\alpha}\sin t \\  -\bar{\beta}\cos t +i\alpha\sin t& \bar{\alpha}\cos t+i\beta\sin t\end{pmatrix},
        \end{align*}
        we have
        \begin{align*}
            \Psi(\exp(tX_1)k) &= (\alpha \cos t-i\bar{\beta}\sin t)(\bar{\beta}\cos t -i\alpha\sin t)+(\bar{\alpha}\cos t+i\beta\sin t)({\beta}\cos t +i\bar{\alpha}\sin t) \\
            &=(\alpha\bar{\beta}+\bar{\alpha}\beta)(\cos^2t-\sin^2 t)-i(\alpha^2-\bar{\alpha}^2-\beta^2+\bar{\beta}^2)\sin t \cos t,
        \end{align*}
        and
        \begin{align*}
            \Theta(\exp(tX_1)k) &= |\alpha \cos t-i\bar{\beta}\sin t|^2-|{\beta}\cos t +i\bar{\alpha}\sin t|^2\\
            &=(|\alpha|^2-|\beta|^2)(\cos^2 t-\sin^2 t)+2i(\alpha\beta-\bar{\alpha}\bar{\beta})\sin t\cos t.
        \end{align*}
        Therefore, $\frac{\partial}{\partial t}\Psi(\exp(tX_1)k)|_{t=0} = -i(\alpha^2-\bar{\alpha}^2-\beta^2+\bar{\beta}^2)$ and $\frac{\partial}{\partial t}\Theta(\exp(tX_1)k)|_{t=0} = 2i(\alpha\beta-\bar{\alpha}\bar{\beta})$.

        The calculations for $X_2$ and $X_3$ are similar.
    \end{proof}

    \begin{proposition}\label{Hessian}
        The Hessian of $\tilde{\phi}(x_1,t_1,x_2,t_2,s_1,s_2,s_3)$ at the critical point $(x_1^\prime,t_1^\prime,x_2^\prime,t_2^\prime,0,0,0)$ is
        \begin{align*}
            D = \begin{pmatrix}D_0&\\&0\end{pmatrix} 
        \end{align*}
        with
        $$
        D_0 = 
            \begin{pmatrix}
           e^{-2t_1^\prime}&0&0&0&-2e^{-t_1^\prime}\sin(2\theta_1)&         2e^{-t_1^\prime}\cos(2\theta_1)      \\
           
           0&1&0&0&2\epsilon_1 \cos(2\theta_1)& 2\epsilon_1 \sin(2\theta_1)\\
           
           0&0&-e^{-2t_2^\prime}&0 &2e^{h-t_2^\prime}\sin(2\theta_2)&        -2e^{h-t_2^\prime}\cos(2\theta_2)       \\ 
           
           0&0&0&-1&  -2\epsilon_2 e^h\cos(2\theta_2)& -2\epsilon_2 e^h\sin(2\theta_2)\\
           
           -2e^{-t_1^\prime}\sin(2\theta_1)&2\epsilon_1 \cos(2\theta_1)&2e^{h-t_2^\prime}\sin(2\theta_2)&-2\epsilon_2 e^h\cos(2\theta_2)& 2(1-e^{2h})&0\\
           
           2e^{-t_1^\prime}\cos(2\theta_1) &2\epsilon_1 \sin(2\theta_1)&        -2e^{h-t_2^\prime}\cos(2\theta_2)   & -2\epsilon_2 e^h\sin(2\theta_2)&0&2(1-e^{2h})
           \end{pmatrix}.
       $$
         Here $D_0$ is the Hessian transversal to $\{ (x_1^\prime,t_1^\prime,x_2^\prime,t_2^\prime,0,0,t): t \text{ varies near }0\}$.
        The determinant of $D_0$ is
        \begin{align*}
            \det(D_0) = 4e^{-2(t_1^\prime+t_2^\prime)} (1-e^{2h})^2.
        \end{align*}
        Recall that $\epsilon_1,\epsilon_2,\theta_1,\theta_2$ are given by (\ref{definition of ujprime}).
    \end{proposition}
    
    \begin{proof}
    It is clear that $\partial^2\tilde{\phi}/\partial x_1\partial 
    x_2$, $\partial^2\tilde{\phi}/\partial x_1\partial t_2$, $\partial^2\tilde{\phi}/\partial t_1\partial x_2$  and $\partial^2\tilde{\phi}/\partial t_1\partial t_2$ are identically $0$. To calculate $\partial^2 \tilde{\phi}/\partial x_1^2$ and   $\partial^2 \tilde{\phi}/\partial x_1\partial t_1$, define the map $v_1:\BR \to K_0$ by the condition that $k^\prime n(x_1^\prime+x)a(t_1^\prime) \in NAv_1(x)$. Note that $v(0) = u_1^\prime$. Lemma \ref{lemma A} gives
        \begin{align*}
            \frac{\partial}{\partial x} \tilde{\phi}(x_1^\prime+x,t_1^\prime,x_2^\prime,t_2^\prime,0,0,0) = -e^{-t_1^\prime}\Psi(v_1(x)),
        \end{align*}
        and
        \begin{align*}
            \frac{\partial}{\partial t_1} \tilde{\phi}(x_1^\prime+x,t_1^\prime,x_2^\prime,t_2^\prime,0,0,0)=\frac{\partial}{\partial t} \tilde{\phi}(x_1^\prime+x,t_1^\prime+t,x_2^\prime,t_2^\prime,0,0,0)|_{t=0} = -\Theta(v_1(x)).
        \end{align*}
        We have, by the definition of the map $v_1(\cdot)$,
        \begin{align*}
            k^\prime n(x_1^\prime)a(t_1^\prime)n(e^{-t_1^\prime}x)  &\in NAv_1(x).
        \end{align*}
        Pluging in $x=0$, or applying (\ref{u1 u2}) with $k=k^\prime$, we will have
        \begin{align*}
           k^\prime n(x_1^\prime)a(t_1^\prime)& \in NA u_1^\prime.
        \end{align*}
        Hence, $ k^\prime n(x_1^\prime)a(t_1^\prime)n(e^{-t_1^\prime}x) $ also lies in the coset $NA u_1^\prime n(e^{-t_1^\prime}x) $. It implies that
        \begin{align*}
             NA u_1^\prime n(e^{-t_1^\prime}x)&=  NAv_1(x),
        \end{align*}
        so that
        \begin{align*}
             u_1^\prime n(e^{-t_1^\prime}x)&\in  NAv_1(x).
        \end{align*}
        By Lemma \ref{Iwasawa k},
        \begin{align*}
            v_1(x)&= \kappa\left( u_1^\prime n(e^{-t_1^\prime}x)  \right)\\
            &= \kappa\left(\begin{pmatrix}e^{i\theta_1}&0\\0&e^{-i\theta_1}\end{pmatrix}\begin{pmatrix}1/\sqrt{2}& \epsilon_1 i/\sqrt{2}\\ \epsilon_1i/\sqrt{2}&1/\sqrt{2}\end{pmatrix} \begin{pmatrix}1&e^{-t_1^\prime}x\\0&1\end{pmatrix}\right)\\
            &=\kappa\left(\frac{1}{\sqrt{2}}\begin{pmatrix}{e^{i\theta_1}} &     (e^{-t_1^\prime }x+\epsilon_1i)e^{i\theta_1}\\\epsilon_1ie^{-i\theta_1}&(\epsilon_1ie^{-t_1^\prime }x+1)e^{-i\theta_1}\end{pmatrix}\right)\\
            &=\frac{1}{\sqrt{e^{-2t_1^\prime }x^2+2} }\begin{pmatrix}                    (-\epsilon_1ie^{-t_1^\prime }x+1)e^{i\theta_1}& \epsilon_1ie^{i\theta_1}\\\epsilon_1ie^{-i\theta_1}& (\epsilon_1ie^{-t_1^\prime }x+1)e^{-i\theta_1}\end{pmatrix},
        \end{align*}
        so
        $$\Psi(v_1(x)) = \frac{-2e^{-t_1^\prime}x}{e^{-2t_1^\prime }x^2+2}, $$
        and
        $$\Theta(v_1(x)) = \frac{e^{-2t_1^\prime}x^2}{e^{-2t_1^\prime }x^2+2}. $$
        Therefore, we have
         \begin{align*}
            \frac{\partial^2}{\partial x_1^2} \tilde{\phi}(x_1^\prime,t_1^\prime,x_2^\prime,t_2^\prime,0,0,0) = \left.\frac{\partial}{\partial x}\left(-e^{-t_1^\prime}\Psi(v_1(x))\right)\right|_{x=0} =\left.\frac{\partial}{\partial x}\left(\frac{2e^{-2t_1^\prime}x}{e^{-2t_1^\prime }x^2+2} \right)\right|_{x=0}=e^{-2t_1^\prime},
        \end{align*}
        and
        \begin{align*}
            \frac{\partial^2}{\partial x_1\partial t_1} \tilde{\phi}(x_1^\prime,t_1^\prime,x_2^\prime,t_2^\prime,0,0,0)=\left.\frac{\partial}{\partial x}\left( -\Theta(v_1(x))\right)\right|_{x=0}=\left.\frac{\partial}{\partial x}\left( \frac{-e^{-2t_1^\prime}x^2}{e^{-2t_1^\prime }x^2+2}\right)\right|_{x=0}=0.
        \end{align*}
        The calculations of $\partial^2 \tilde{\phi}/\partial x_2^2$ and   $\partial^2 \tilde{\phi}/\partial x_2\partial t_2$ are identical. We may define the map $v_2:\BR\to K_0$ by the condition $k^\prime g n(x_2^\prime +x )a(t_2^\prime)\in NAv_2(x)$. By the same calculation as above, we may obtain that
        $  v_2(x)= \kappa\left( u_2^\prime n(e^{-t_2^\prime}x)  \right)$ and 
        \begin{align*}
           & \frac{\partial}{\partial x} \tilde{\phi}(x_1^\prime,t_1^\prime,x_2^\prime+x,t_2^\prime,0,0,0) = e^{-t_2^\prime}\Psi(v_2(x)),\\
            &\frac{\partial}{\partial t_2} \tilde{\phi}(x_1^\prime,t_1^\prime,x_2^\prime+x,t_2^\prime,0,0,0)= \Theta(v_2(x)),
        \end{align*}
        and so
        \begin{align*}
            &\frac{\partial^2}{\partial x_2^2} \tilde{\phi}(x_1^\prime,t_1^\prime,x_2^\prime,t_2^\prime,0,0,0)  = -e^{-2t_2^\prime},\\
            &\frac{\partial^2}{\partial x_2\partial t_2} \tilde{\phi}(x_1^\prime,t_1^\prime,x_2^\prime,t_2^\prime,0,0,0)=0.
        \end{align*}

        To calculate  $\partial^2 \tilde{\phi}/\partial t_1^2$ and $\partial^2 \tilde{\phi}/\partial t_2^2$, we  define maps $w_1,w_2:\BR \to K_0$ by the conditions that
        \begin{align}\label{def of w1 w2}
            k^\prime n(x_1^\prime)a(t_1^\prime+t) \in NAw_1(t)\quad\text{ and }\quad  k^\prime g n(x_2^\prime)a(t_2^\prime+t) \in NAw_2(t).
        \end{align}
        Notice that $w_1(0)=u_1^\prime$ and $w_2(0)=u_2^\prime$.
         Lemma \ref{lemma A} gives
        \begin{align*}
            &\frac{\partial}{\partial t} \tilde{\phi}(x_1^\prime,t_1^\prime+t,x_2^\prime,t_2^\prime,0,0,0) = -\Theta(w_1(t)),\\
            &\frac{\partial}{\partial t} \tilde{\phi}(x_1^\prime,t_1^\prime,x_2^\prime,t_2^\prime+t,0,0,0) = \Theta(w_2(t)).
        \end{align*}
        Since
        \begin{align*}
            k^\prime n(x_1^\prime)a(t_1^\prime) \in NA u_1^\prime\quad\text{ and }\quad k^\prime gn(x_2^\prime)a(t_2^\prime) \in NA u_2^\prime,
        \end{align*}
        (\ref{def of w1 w2}) provides
        \begin{align*}
            NA u_1^\prime a(t)=  NAw_1(t)\quad&\text{ and } \quad NA u_2^\prime a(t)=  NAw_2(t),
        \end{align*}
        so that
        \begin{align*}
            u_1^\prime a(t)\in  NAw_1(t)\quad&\text{ and } \quad u_2^\prime a(t)\in  NAw_2(t).
        \end{align*}
       By Lemma \ref{Iwasawa k}, we have
        \begin{align*}
            w_1(t)&= \kappa\left(u_1^\prime a(t) \right)\\
            &=\kappa\left(\begin{pmatrix}e^{i\theta_1}&0\\0&e^{-i\theta_1}\end{pmatrix}\begin{pmatrix}1/\sqrt{2}& \epsilon_1 i/\sqrt{2}\\ \epsilon_1i/\sqrt{2}&1/\sqrt{2}\end{pmatrix} \begin{pmatrix}e^{t/2}&0\\0&e^{-t/2}\end{pmatrix}\right)\\
            &=\kappa\left( \frac{1}{\sqrt{2}}\begin{pmatrix}e^{t/2+i\theta_1}& e^{-t/2+i\theta_1}\epsilon_1 i\\ e^{t/2-i\theta_1}\epsilon_1i&e^{-t/2-i\theta_1}\end{pmatrix}\right)\\
            &=\frac{1}{\sqrt{e^t+e^{-t}}}\begin{pmatrix} e^{-t/2+i\theta_1}&e^{t/2+i\theta_1}\epsilon_1i\\e^{t/2-i\theta_1}\epsilon_1i&e^{-t/2-i\theta_1} \end{pmatrix},
        \end{align*}
        and
        \begin{align*}
            w_2(t) = \frac{1}{\sqrt{e^t+e^{-t}}}\begin{pmatrix} e^{-t/2+i\theta_2}&e^{t/2+i\theta_2}\epsilon_2i\\e^{t/2-i\theta_2}\epsilon_2i&e^{-t/2-i\theta_2} \end{pmatrix},
        \end{align*}
        and so
        \begin{align*}
            \Theta(w_1(t)) =\Theta(w_2(t)) = \frac{e^{-t}-e^t}{e^{-t}+e^t}.
        \end{align*}
        Therefore, we have
         \begin{align*}
            \frac{\partial^2}{\partial t_1^2} \tilde{\phi}(x_1^\prime,t_1^\prime,x_2^\prime,t_2^\prime,0,0,0)=\left. \frac{\partial}{\partial t}\left(-\Theta(w_1(t))\right)\right|_{t=0}=\left. \frac{\partial}{\partial t}\left(\frac{e^{t}-e^{-t}}{e^{t}+e^{-t}}\right)\right|_{t=0}=1,
        \end{align*}
        and $$ \frac{\partial^2}{\partial t_2^2} \tilde{\phi}(x_1^\prime,t_1^\prime,x_2^\prime,t_2^\prime,0,0,0)=-1.$$
        
        To calculate $\partial^2 \tilde{\phi}/\partial s_i\partial x_1$ and $\partial^2 \tilde{\phi}/\partial s_i\partial t_1$, by Lemma \ref{lemma A}, we have
        \begin{align*}
            \frac{\partial}{\partial x_1} \tilde{\phi}(x_1^\prime,t_1^\prime,x_2^\prime,t_2^\prime,s_1,s_2,s_3) &= \left. \frac{\partial}{\partial x}\left(- A\left( u_1^{-1}(\exp(s_1X_1+s_2X_2+s_3X_3) u_1^\prime)n(x_1^\prime+x)a(t_1^\prime) \right) \right)\right|_{x=0}\\
            &=\left. \frac{\partial}{\partial x}\left(- A\left( u_1^{-1}(\exp(s_1X_1+s_2X_2+s_3X_3) u_1^\prime)n(x_1^\prime)a(t_1^\prime) n(e^{-t_1^\prime}x)\right) \right)\right|_{x=0}\\
            &=-e^{-t_1^\prime}\Psi\circ\kappa \left(  u_1^{-1}(\exp(s_1X_1+s_2X_2+s_3X_3) u_1^\prime) n(x_1^\prime)a(t_1^\prime) \right)\\
            &=-e^{-t_1^\prime}\Psi(\exp(s_1X_1+s_2X_2+s_3X_3)u_1^\prime).
        \end{align*}
        The last equality holds because, if we let $k=u_1^{-1}(\exp(s_1X_1+s_2X_2+s_3X_3) u_1^\prime)$ in \eqref{u1 u2}, then
        \begin{align*}
            u_1^{-1}(\exp(s_1X_1+s_2X_2+s_3X_3) u_1^\prime) n(x_1^\prime)a(t_1^\prime) \in NA \exp(s_1X_1+s_2X_2+s_3X_3) u_1^\prime.
        \end{align*}
        Similarly, we have
        \begin{align*}
            \frac{\partial}{\partial t_1} \tilde{\phi}(x_1^\prime,t_1^\prime,x_2^\prime,t_2^\prime,s_1,s_2,s_3) = -\Theta(\exp(s_1X_1+s_2X_2+s_3X_3)u_1^\prime).
        \end{align*}
         Lemma \ref{gradient of psi theta} implies that
        \begin{align*}
            \left(\frac{\partial^2}{\partial s_1\partial x_1},\frac{\partial^2}{\partial s_2\partial x_1},\frac{\partial^2}{\partial s_3\partial x_1} \right)\tilde{\phi}(x_1^\prime,t_1^\prime,x_2^\prime,t_2^\prime,0,0,0) =\left(                 -2e^{-t_1^\prime}\sin(2\theta_1),         2e^{-t_1^\prime}\cos(2\theta_1),        0  \right),
        \end{align*}
        and
        \begin{align*}
            \left(\frac{\partial^2}{\partial s_1\partial t_1},\frac{\partial^2}{\partial s_2\partial t_1},\frac{\partial^2}{\partial s_3\partial t_1} \right)\tilde{\phi}(x_1^\prime,t_1^\prime,x_2^\prime,t_2^\prime,0,0,0) =\left( 2\epsilon_1 \cos(2\theta_1), 2\epsilon_1 \sin(2\theta_1), 0\right).
        \end{align*}
        
        Now we calculate $\partial^2 \tilde{\phi}/\partial s_i\partial x_2$ and $\partial^2 \tilde{\phi}/\partial s_i\partial t_2$.  From the commutative diagram (\ref{commutative diagram}), we have
        \begin{align*}
            u_2\left(  u_1^{-1}(\exp(s_1X_1+s_2X_2+s_3X_3) u_1^\prime)   \right) = \exp(\alpha_{12}(s_1,s_2,s_3)) u_2^\prime = \exp(r_1X_1+r_2X_2+r_3X_3) u_2^\prime
        \end{align*}
        for $(s_1,s_2,s_3)\in U_1$. We recall that $kgn(x_2^\prime)a(t_2^\prime) \in NA u_2(k)$ for any $k\in K_0$.  Hence, if we let $k =u_1^{-1}(\exp(s_1X_1+s_2X_2+s_3X_3) u_1^\prime) $, then
        \begin{align}
        \label{change of variable for k}
             u_1^{-1}(\exp(s_1X_1+s_2X_2+s_3X_3) u_1^\prime)gn(x_2^\prime)a(t_2^\prime) \in NA \exp(r_1X_1+r_2X_2+r_3X_3) u_2^\prime.
        \end{align}
        Hence,
        \begin{align*}
            \frac{\partial}{\partial x_2} \tilde{\phi}(x_1^\prime,t_1^\prime,x_2^\prime,t_2^\prime,s_1,s_2,s_3) &=  \left.\frac{\partial}{\partial x}\left( A\left( u_1^{-1}(\exp(s_1X_1+s_2X_2+s_3X_3) u_1^\prime)gn(x_2^\prime+x)a(t_2^\prime) \right) \right)\right|_{x=0}\\
            &= \left.\frac{\partial}{\partial x}\left( A\left( u_1^{-1}(\exp(s_1X_1+s_2X_2+s_3X_3) u_1^\prime)n(x_2^\prime)a(t_2^\prime) n(e^{-t_1^\prime}x)\right) \right)\right|_{x=0}\\
            &=e^{-t_2^\prime}\Psi\circ\kappa \left( u_1^{-1}(\exp(s_1X_1+s_2X_2+s_3X_3) u_1^\prime)gn(x_2^\prime)a(t_2^\prime)\right)\\
            &=e^{-t_2^\prime}\Psi\left(\exp(r_1X_1+r_2X_2+r_3X_3) u_2^\prime\right).
        \end{align*}
       Applying the chain rule,  Lemma \ref{gradient of psi theta}, and Proposition \ref{jacobian u1 u2}, we have
        \begin{align*}
            &\left(\frac{\partial^2}{\partial s_1\partial x_2},\frac{\partial^2}{\partial s_2\partial x_2},\frac{\partial^2}{\partial s_3\partial x_2} \right)\tilde{\phi}(x_1^\prime,t_1^\prime,x_2^\prime,t_2^\prime,0,0,0) \\
            =& \left(\frac{\partial}{\partial r_1},\frac{\partial}{\partial r_2},\frac{\partial}{\partial r_3} \right)   \frac{\partial \tilde{\phi}}{\partial x_2}(x_1^\prime,t_1^\prime,x_2^\prime,t_2^\prime,0,0,0) \cdot  \bJ\alpha_{12}(0,0,0)\\
            =&\left(                 2e^{-t_2^\prime}\sin(2\theta_2),         -2e^{-t_2^\prime}\cos(2\theta_2),        0  \right)\begin{pmatrix}e^h& & \\&e^h&\\&&1 \end{pmatrix}\\
            =&\left(                 2e^{h-t_2^\prime}\sin(2\theta_2),         -2e^{h-t_2^\prime}\cos(2\theta_2),        0  \right),
        \end{align*}
        and similarly
        \begin{align*}
            &\left(\frac{\partial^2}{\partial s_1\partial t_2},\frac{\partial^2}{\partial s_2\partial t_2},\frac{\partial^2}{\partial s_3\partial t_2} \right)\tilde{\phi}(x_1^\prime,t_1^\prime,x_2^\prime,t_2^\prime,0,0,0)\\
            =&\left( -2\epsilon_2 \cos(2\theta_2), -2\epsilon_2 \sin(2\theta_2), 0\right)\begin{pmatrix}e^h& & \\&e^h&\\&&1 \end{pmatrix}\\
            =&\left( -2\epsilon_2 e^h\cos(2\theta_2), -2\epsilon_2 e^h\sin(2\theta_2), 0\right).
        \end{align*}
        
        To calculate  $\partial^2 \tilde{\phi}/\partial s_i\partial s_j$, by applying Lemma \ref{spliting A} and the condition (\ref{definition of h}), we obtain
        \begin{align*}
            &\tilde{\phi}(x_1^\prime,t_1^\prime,x_2^\prime,t_2^\prime,s_1,s_2,s_3)\\
            =& - A\left( \Phi_{n(x_1^\prime)a(t_1^\prime)}^{-1}(\exp(s_1X_1+s_2X_2+s_3X_3) u_1^\prime)n(x_1^\prime)a(t_1^\prime) \right) \\
            &+ A\left( \Phi_{n(x_1^\prime)a(t_1^\prime)}^{-1}(\exp(s_1X_1+s_2X_2+s_3X_3) u_1^\prime)gn(x_2^\prime)a(t_2^\prime) \right)\\
            =&A\left(\exp(s_1X_1+s_2X_2+s_3X_3) u_1^\prime a(-t_1^\prime)n(-x_1^\prime)g n(x_2^\prime)a(t_2^\prime)\right)\\
            =& A\left(\exp(s_1X_1+s_2X_2+s_3X_3) a(h)\right).
        \end{align*}
        This gives, by Lemma \ref{lemma A},
        \begin{align*}
            \frac{\partial^2}{\partial s_1^2}\tilde{\phi}(x_1^\prime,t_1^\prime,x_2^\prime,t_2^\prime,0,0,0) =  \left.\frac{\partial^2}{\partial t^2} A\left(\exp(t X_1) a(h)\right)\right|_{t=0} = \left.\frac{\partial^2}{\partial t^2} \left(h-\log(\cos^2t+e^{2h}\sin^2t)\right)\right|_{t=0} =2(1-e^{2h}).
        \end{align*}
        Following similar calculations, we obtain
        $\frac{\partial^2}{\partial s_2^2}\tilde{\phi}(x_1^\prime,t_1^\prime,x_2^\prime,t_2^\prime,0,0,0) =2(1-e^{2h}) $, $\frac{\partial^2}{\partial s_3^2}\tilde{\phi}(x_1^\prime,t_1^\prime,x_2^\prime,t_2^\prime,0,0,0) =0$ and $\frac{\partial^2}{\partial s_i\partial s_j}\tilde{\phi}(x_1^\prime,t_1^\prime,x_2^\prime,t_2^\prime,0,0,0) =0$ with $i\neq j$.
    \end{proof}

    \subsection{The functions $\psi$ and $\tilde{\psi}$}\label{sec degenerate case}
    The computation of the transversal Hessian $D_0$ shows that although $D_0$ is non-degenerate if $g \notin H^\prime$, its determinant tends to $0$ as $h\to 0$, so the bound of $ J(s,g;x_1^\prime,t_1^\prime,x_2^\prime,t_2^\prime,k^\prime)$ given by stationary phase can not be applied for $h$ near $0$, or more generally when $d(g,H^\prime)$ is near $0$. To deal with this issue, we will first eliminate the variables $x_1,t_1,x_2,t_2$ because if $(x_1,t_1,x_2,t_2,k,g)$ is a critical point of $\phi$, then $x_1,t_1,x_2,t_2$ can be uniquely determined when $k$ and $g$ are fixed. After eliminating the variables $x_1,t_1,x_2,t_2$, we reduce the phase function $\phi$ to a new function $\psi$ and show that $\psi$ has a Hessian that behaves almost the same as $D$ at the critical points.
    
    Define $\cP =  K_0\times G_0$.
    Note that the image of $g\BH^2$  for $g\in G_0$ in the upper half space model of $\BH^3$ must be in one of the following two cases:
    \begin{enumerate}
        \item $g\BH^2$ is a vertical plane. More precisely, it is
        \begin{align*}
            \{(z,t)\in\BH^3\,|\,z\in L,t>0\},\quad\text{ for some real affine line }L\subset\BC.
        \end{align*}
        \item $g\BH^2$ is a hemisphere orthogonal to the boundary $\{(z,0)\,|\,z\in\BC\}$.
    \end{enumerate}
    See e.g. \cite[Proposition A.5.6]{MR1219310} for the above result.
    We define $\cS\subset\cP $ to be the set consisting of $(k,g)\in\cP$ so that at least one of $k\BH^2$ and $kg\BH^2$ is a vertical plane. Note that $\cS$ is a subset of $\cP$ with measure $0$. We define functions
    \begin{align*}
        \eta_1,\xi_1,\eta_2,\xi_2:\cP\backslash\cS\to\BR
    \end{align*}
    by requiring that $k. \left(\eta_1(k,g), e^{\xi_1(k,g)}\right)$ is the unique point on $k\BH^2$ with the highest $A$, i.e. it is the north pole of the hemisphere, and likewise for $\eta_2(k,g)$ and $\xi_2(k,g)$ for $kg\BH^2$. Note that as $\eta_1$ and $\xi_1$ do not depend on $g$, we will omit this argument of the function. It follows from Proposition \ref{crtical for x} that
    \begin{align}\label{50}
        kn\left( \eta_1(k) \right) a\left(\xi_1(k)\right) \in NA\begin{pmatrix}e^{i\tau_1(k)}&0\\0&e^{-i\tau_1(k)}\end{pmatrix}\begin{pmatrix}1/\sqrt{2}& \sigma_1(k)i/\sqrt{2}\\ \sigma_1(k)i/\sqrt{2}&1/\sqrt{2}\end{pmatrix},
    \end{align}
    and
    \begin{align}\label{50'}
        kgn\left( \eta_2(k,g) \right) a\left(\xi_2(k,g)\right) \in NA\begin{pmatrix}e^{i\tau_2(k,g)}&0\\0&e^{-i\tau_2(k,g)}\end{pmatrix}\begin{pmatrix}1/\sqrt{2}& \sigma_2(k,g)i/\sqrt{2}\\\sigma_2(k,g)i/\sqrt{2}&1/\sqrt{2}\end{pmatrix}.
    \end{align}
    Here $\tau_1$, $\tau_2$ are smooth functions valued in $\BR/2\pi\BZ$ and  $\sigma_1,\sigma_2$ are locally constant functions valued in $\{\pm1\}$.
    Equivalently, by the definition (\ref{defn of phi}) and Proposition \ref{crtical for x}, the functions $\eta_1$, $\xi_1$, $\eta_2$ and $\xi_2$ may also be characterized as the unique functions satisfying
    \begin{align}\label{51}
        \frac{\partial \phi}{\partial x_1}(\eta_1(k),\xi_1(k),x_2,t_2,k,g) =  \frac{\partial \phi}{\partial t_1}(\eta_1(k),\xi_1(k),x_2,t_2,k,g) =0,
    \end{align}
    and
    \begin{align}\label{51'}
        \frac{\partial \phi}{\partial x_2}(x_1,t_1,\eta_2(k,g),\xi_2(k,g),g) =  \frac{\partial \phi}{\partial t_2}(x_1,t_1,\eta_2(k,g),\xi_2(k,g),g) = 0.
    \end{align}
    We define
    \begin{align}\label{defn of psi}
    \begin{split}
        &\psi:\cP\backslash\cS \to \BR\\
        &\psi(k,g) = \phi(\eta_1(k),\xi_1(k),\eta_2(k,g),\xi_2(k,g),k,g).
        \end{split}
    \end{align}
    It is clear that $g\in H^\prime$ implies $\psi(k,g)=0$.
    Combing (\ref{51}) and (\ref{51'}) with Proposition \ref{critical points}, we get the following lemma.
    \begin{lemma}\label{lem critical pt of psi}
        The point $(k^\prime,g^\prime)\in\cP\backslash\cS$ is a critical point of $\psi$ with respect to the variable $k$ exactly when $(\eta_1(k^\prime),\xi_1(k^\prime),\eta_2(k^\prime,g^\prime),\xi_2(k^\prime,g^\prime),k^\prime,g^\prime)$ is a critical point of $\phi$ with respect to $(x_1,t_1,x_2,t_2,k)$.
    \end{lemma}
    
    Fix a $k^\prime\in K_0$. We temporarily do not assume $k^\prime$ is a critical point.
     We define $u(\cdot) = \Phi_{n(\eta_1(k^\prime))a(\xi_1(k^\prime))}(\cdot)$ to be the map from $K_0$ to itself so that
    \begin{align}\label{defn of Phi k prime}
        kn(\eta_1(k^\prime))a(\xi_1(k^\prime))\in NA u(k).
    \end{align}
    The map $u$ is a diffeomorphism by Lemma \ref{change of variable}.
    
    \begin{proposition}\label{derivatives of eta xi wrt si}
        Let $k^\prime\in K_0$ be fixed and define
        \begin{align*}
            k(s_1,s_2,s_3) = u^{-1}\left(\exp(s_1X_1+s_2X_2+s_3X_3)u(k^\prime)\right),
        \end{align*}
        $ \eta_1(s_1,s_2,s_3) = \eta_1\left( k(s_1,s_2,s_3)\right)$, and $\xi_1(s_1,s_2,s_3) = \xi_1\left( k(s_1,s_2,s_3) \right)$, then
        \begin{align*}
                \begin{pmatrix}
                    \partial \eta_1/\partial s_1& \partial \eta_1/\partial s_2&\partial \eta_1/\partial s_3\\
                   \partial \xi_1/\partial s_1& \partial \xi_1/\partial s_2&\partial \xi_1/\partial s_3
                \end{pmatrix}(0,0) = \begin{pmatrix}
                      2e^{\xi_1(k^\prime)} \sin(2\theta_1)&-2e^{\xi_1(k^\prime)} \cos(2\theta_1)&0\\
                   -2\epsilon_1\cos(2\theta_1)&-2\epsilon_1\sin(2\theta_1)&0
                \end{pmatrix},
            \end{align*}
            where $\theta_1 = \tau_1(k^\prime)$ and $\epsilon_1 = \sigma_1(k^\prime)$.
    
            Hence, the map $k\mapsto (\eta_1(k),\xi_1(k))$ from $\mathrm{U}(1)\backslash K_0$ to $\BR^2$ is a local diffeomorphism.
    \end{proposition}
    \begin{proof}
         If we let $k = k(s_1,s_2,s_3)$ in (\ref{50}), then
        \begin{align}\label{k s1 eta xi 1}
            \begin{split}
                &k(s_1,s_2,s_3)n\left(\eta_1(s_1,s_2,s_3)\right)a\left(\xi_1(s_1,s_2,s_3)\right) \\
            \in& NA \begin{pmatrix}e^{i\tau_1(k(s_1,s_2,s_3))}&0\\0&e^{-i\tau_1(k(s_1,s_2,s_3))}\end{pmatrix}\begin{pmatrix}1/\sqrt{2}& \sigma_1(k(s_1,s_2,s_3))i/\sqrt{2}\\ \sigma_1(k(s_1,s_2,s_3))i/\sqrt{2}&1/\sqrt{2}\end{pmatrix}.
            \end{split}
        \end{align}
        If we let $k=k(s_1,s_2,s_3)$ in (\ref{defn of Phi k prime}) and let $x_1^\prime = \eta_1(k^\prime), t_1^\prime = \xi_1(k^\prime)$, then
        \begin{align}\label{k s1 x1t1prime}
            k(s_1,s_2,s_3)n\left(x_1^\prime\right)a\left(t_1^\prime\right) \in  NA u(k(s_1,s_2,s_3)) =  NA \exp(s_1X_1+s_2X_2+s_3X_3) u(k^\prime).
        \end{align}
        Here the notations $x_1^\prime, t_1^\prime, k^\prime$ do not mean they form a critical point, which differs from the previous subsection. Following from Proposition \ref{crtical for x},
        \begin{align*}
            u(k^\prime) = \begin{pmatrix}e^{i\theta_1}&0\\0&e^{-i\theta_1}\end{pmatrix}\begin{pmatrix}1/\sqrt{2}& \epsilon_1i/\sqrt{2}\\\epsilon_1i/\sqrt{2}&1/\sqrt{2}\end{pmatrix},
        \end{align*}
        with $\theta_1 = \tau_1(k^\prime)$ and $\epsilon_1 = \sigma_1(k^\prime)$. Since $\sigma_1(k(0,0,0)) = \sigma_1(k^\prime) = \epsilon_1$ and the function $\sigma_1$ is locally constant, it may be seen $\sigma_1(k(s_1,s_2,s_3)) = \epsilon_1$ for $s_i$ close to $0$.
        Hence,
        \begin{align*}
            &\exp(s_1X_1+s_2X_2+s_3X_3)\begin{pmatrix}e^{i\theta_1}&0\\0&e^{-i\theta_1}\end{pmatrix}\begin{pmatrix}1/\sqrt{2}& \epsilon_1i/\sqrt{2}\\\epsilon_1i/\sqrt{2}&1/\sqrt{2}\end{pmatrix}n\left(e^{-t_1^\prime}(\eta_1(s_1,s_2,s_3)-x_1^\prime)\right)a\left(\xi_1(s_1,s_2,s_3)-t_1^\prime\right)\\
            =&\exp(s_1X_1+s_2X_2+s_3X_3) u(k^\prime) n\left(e^{-t_1^\prime}(\eta_1(s_1,s_2,s_3)-x_1^\prime)\right)a(-t_1^\prime)a\left(\xi_1(s_1,s_2,s_3)\right)\\
            =&\exp(s_1X_1+s_2X_2+s_3X_3) u(k^\prime) a(-t_1^\prime)n(-x_1^\prime)n\left(\eta_1(s_1,s_2,s_3)\right)a\left(\xi_1(s_1,s_2,s_3)\right)\\
            =& \exp(s_1X_1+s_2X_2+s_3X_3) u(k^\prime) \cdot\left(k(s_1,s_2,s_3)n(x_1^\prime) a(t_1^\prime)\right)^{-1}\cdot\left(k(s_1,s_2,s_3)n\left(\eta_1(s_1,s_2,s_3)\right)a\left(\xi_1(s_1,s_2,s_3)\right) \right)\\
            \in& NA \begin{pmatrix}e^{i\tau_1(k(s_1,s_2,s_3))}&0\\0&e^{-i\tau_1(k(s_1,s_2,s_3))}\end{pmatrix}\begin{pmatrix}1/\sqrt{2}& \epsilon_1i/\sqrt{2}\\\epsilon_1i/\sqrt{2}&1/\sqrt{2}\end{pmatrix}.
        \end{align*}
        The last line above holds because of (\ref{k s1 eta xi 1}) and (\ref{k s1 x1t1prime}).
        Taking the inverse gives
        \begin{align}\label{belong infinity}
        \begin{split}
            a\left(-(\xi_1(s_1,s_2,s_3)-t_1^\prime)\right)n\left(-e^{-t_1^\prime}(\eta_1(s_1,s_2,s_3)-x_1^\prime)\right)\begin{pmatrix}1/\sqrt{2}& -\epsilon_1i/\sqrt{2}\\-\epsilon_1i/\sqrt{2}&1/\sqrt{2}\end{pmatrix}\\
            \begin{pmatrix}e^{-i\theta_1}&0\\0&e^{i\theta_1}\end{pmatrix}\exp(-s_1X_1-s_2X_2-s_3X_3) \in\begin{pmatrix}1/\sqrt{2}& -\epsilon_1i/\sqrt{2}\\-\epsilon_1i/\sqrt{2}&1/\sqrt{2}\end{pmatrix}\operatorname{U}(1)AN.
        \end{split}
        \end{align}
        If we let both sides of \eqref{belong infinity} act on the point $\infty$ in $\BH^3$ with $s_2=s_3=0$, we obtain
        \begin{align*}
            e^{-(\xi_1(s_1,0,0)-t_1^\prime)}\left(\frac{ie^{-2i\theta_1}\cos s_1-\epsilon_1i\sin s_1}{\epsilon_1 e^{-2i\theta_1}\cos s_1 +\sin s_1} -e^{-t_1^\prime}(\eta_1(s_1,0,0)-x_1^\prime)   \right)= \epsilon_1 i.
        \end{align*}
        Taking the derivative with respect to $s_1$ at $0$ gives
        \begin{align*}
            \epsilon_1 i\frac{\partial\xi_1}{\partial s_1}(0,0,0) +2i e^{2i\theta_1}+ e^{-t_1^\prime}\frac{\partial\eta_1}{\partial s_1}(0,0,0) = 0.
        \end{align*}
        Since $\eta_1,\xi_1$ are real-valued,
        \begin{align*}
            \frac{\partial\eta_1}{\partial s_1}(0,0,0) =2e^{t_1^\prime} \sin(2\theta_1) \quad\text{ and }\quad
            \frac{\partial\xi_1}{\partial s_1}(0,0,0) =-2\epsilon_1\cos(2\theta_1).
        \end{align*}
        
        To calculate $\partial \eta_1/\partial s_2(0,0,0)$, $\partial \eta_1/\partial s_3(0,0,0)$, $\partial \xi_1/\partial s_2(0,0,0)$ and $\partial \xi_1/\partial s_3(0,0,0)$, we let both sides of (\ref{belong infinity}) act on $\infty$ with $s_1 = s_3=0$ and with $s_1=s_2=0$. Then we obtain
        \begin{align*}
            e^{-(\xi_1(0,s_2,0)-t_1^\prime)}\left(\frac{-e^{-2i\theta_1}\cos s_2-\epsilon_1i\sin s_2}{\epsilon_1i e^{-2i\theta_1}\cos s_2 +\sin s_2} -e^{-t_1^\prime}(\eta_1(0,s_2,0)-x_1^\prime)   \right)= \epsilon_1 i,
        \end{align*}
        and
        \begin{align*}
            e^{-(\xi_1(0,0,s_3)-t_1^\prime)}\left(\epsilon_1 i -e^{-t_1^\prime}(\eta_1(0,0,s_3)-x_1^\prime)   \right)= \epsilon_1 i.
        \end{align*}
        Then taking the derivatives with respect to $s_2$ and $s_3$ respectively at $0$ gives
        \begin{align*}
             \frac{\partial\eta_1}{\partial s_2}(0,0,0) = -2e^{t_1^\prime} \cos(2\theta_1),\\
             \frac{\partial\xi_1}{\partial s_2}(0,0,0)=-2\epsilon_1\sin(2\theta_1),\\
             \frac{\partial\eta_1}{\partial s_3}(0,0,0) = \frac{\partial\xi_1}{\partial s_3}(0,0,0)=0.
        \end{align*}
    \end{proof}

    Later when we say $(k^\prime,g^\prime)$ is a critical point of $\psi$, it always means a critical point with respect to $k$-variable, i.e., $k^\prime$ is a critical point of $\psi(\cdot,g^\prime)$.
     Suppose $(k^\prime,g^\prime)$ is a critical point of $\psi$. As the previous subsection, we may look at $\psi$ under the local chart to study its Hessian.
    Define
    \begin{align}\label{tilde psi}
        \tilde{\psi}(s_1,s_2,s_3) = \psi\left( u^{-1}\left(\exp(s_1X_1+s_2X_2+s_3X_3)u(k^\prime)\right),g^\prime \right),
    \end{align}
    so $\tilde{\psi}$ is the expression of $\psi$ near the critical point $(k^\prime,g^\prime)$ under exponential map composed with $u(\cdot)$, and $(0,0,0)$ is a critical point of $\tilde{\psi}$.
    Recall our notation in Section \ref{subsection hessians}. We defined $x_1^\prime = \eta_1(k^\prime)$, $t_1^\prime = \xi_1(k^\prime)$, $x_2^\prime = \eta_2(k^\prime,g^\prime)$, $t_2^\prime = \xi_2(k^\prime,g^\prime)$, and $\theta_j\in\BR/2\pi\BZ$, $\epsilon_j\in\{\pm 1\}$ ($j=1,2$) by
    \begin{align*}
        k^\prime n(x_1^\prime) a(t_1^\prime) \in NA \begin{pmatrix}e^{i\theta_1}&0\\0&e^{-i\theta_1}\end{pmatrix}\begin{pmatrix}1/\sqrt{2}& \epsilon_1 i/\sqrt{2}\\ \epsilon_1i/\sqrt{2}&1/\sqrt{2}\end{pmatrix},\\
        k^\prime g^\prime n(x_2^\prime) a(t_2^\prime) \in NA \begin{pmatrix}e^{i\theta_2}&0\\0&e^{-i\theta_2}\end{pmatrix}\begin{pmatrix}1/\sqrt{2}& \epsilon_2 i/\sqrt{2}\\ \epsilon_2i/\sqrt{2}&1/\sqrt{2}\end{pmatrix}.
    \end{align*}
    Then (\ref{50}) and (\ref{50'}) give the relations $\theta_1 = \tau_1(k^\prime)$, $\epsilon_1 = \sigma_1(k^\prime)$, $\theta_2 = \tau_2(k^\prime,g^\prime)$ and $\epsilon_2 = \sigma_2(k^\prime,g^\prime)$.
    
    \begin{corollary}\label{Hessian of psi}
        The Hessian of $\tilde{\psi}$ at $(0,0,0)$ is
        $
            D^\prime = \begin{pmatrix}
                   D^\prime_0&\\&0
            \end{pmatrix}
        $
        with
        \begin{align*}
            D_0^\prime = \begin{pmatrix}
                   -2(1-e^{2h}) &\\
                   &-2(1-e^{2h})
            \end{pmatrix}.
        \end{align*}
        The determinant of $D_0^\prime$ is
        \begin{align*}
            \det(D_0^\prime) = 4(1-e^{2h})^2.
        \end{align*}
    \end{corollary}
    
    \begin{proof}
       We define, as in Proposition \ref{derivatives of eta xi wrt si},
        \begin{align*}
            k(s_1,s_2,s_3) = u^{-1}\left(\exp(s_1X_1+s_2X_2+s_3X_3)u(k^\prime)\right),
        \end{align*}
        and define
        \begin{align}\label{definition of eta xi s1s2s3}
        \begin{split}
            \eta_1(s_1,s_2,s_3) = \eta_1\left( k(s_1,s_2,s_3)\right) ,\quad&\quad\quad\xi_1(s_1,s_2,s_3) = \xi_1\left( k(s_1,s_2,s_3) \right),\\
            \eta_2(s_1,s_2,s_3) = \eta_2\left(k(s_1,s_2,s_3),g^\prime\right), \quad&\quad\quad \xi_2(s_1,s_2,s_3) = \xi_2\left(k(s_1,s_2,s_3),g^\prime\right) .
            \end{split}
        \end{align}
         Notice that if we let $\tilde{\phi}$ be as in (\ref{phi tilde}), then
        \begin{align*}
            \tilde{\psi}(s_1,s_2,s_3)  = \tilde{\phi}(\eta_1,\xi_1,\eta_2,\xi_2,s_1,s_2,s_3),
        \end{align*}
        with $\eta_i = \eta_i(s_1,s_2,s_3)$ and $\xi_i = \xi_i(s_1,s_2,s_3)$ ($i=1,2$).
        Then
        \begin{align*}
            (x_1^\prime,t_1^\prime,x_2^\prime,t_2^\prime,0,0,0) = (\eta_1(k^\prime),\xi_1(k^\prime),\eta_2(k^\prime),\xi_2(k^\prime),0,0,0)
        \end{align*}
        is a critical point of $\tilde{\phi}$.
         Suppose that $D$ is the Hessian of  $\tilde{\phi}$ at the critical point $(x_1^\prime,t_1^\prime,x_2^\prime,t_2^\prime,0,0,0)$ which is calculated in Proposition \ref{Hessian}.
        If we apply the chain rule to $\tilde{\psi}$ with the functions $\eta_i,\xi_i$ in (\ref{definition of eta xi s1s2s3}), we obtain
        \begin{align*}
            D^\prime = B^tD B
        \end{align*}
        where
        \begin{align*}
            B=\begin{pmatrix}
                   \partial \eta_1/\partial s_1& \partial \eta_1/\partial s_2& \partial \eta_1/\partial s_3\\
                   \partial \xi_1/\partial s_1& \partial \xi_1/\partial s_2& \partial \xi_1/\partial s_3\\
                   \partial \eta_2/\partial s_1& \partial \eta_2/\partial s_2& \partial \eta_2/\partial s_3\\
                   \partial \xi_2/\partial s_1& \partial \xi_2/\partial s_2& \partial \xi_2/\partial s_3\\
                   1&0&0\\
                   0&1&0\\
                   0&0&1
            \end{pmatrix}(0,0,0).
        \end{align*}
        Hence, it suffices to calculate $B$.
        We have already calculated $\partial \eta_1/\partial s_i(0,0,0)$ and $\partial \xi_1/\partial s_i(0,0,0)$ in Proposition \ref{derivatives of eta xi wrt si}.
    
        Recall that we have a change of coordinates $(r_1,r_2,r_3) = \alpha_{12}(s_1,s_2,s_3)$ satisfying (\ref{commutative diagram}).  As in (\ref{definition of eta xi s1s2s3}), we may denote by $k(r_1, r_2, r_3), \eta_2 (r_1,r_2,r_3),\xi_2 (r_1,r_2,r_3)$ the corresponding functions in the $(r_1, r_2, r_3)$ variables.  Then, as in (\ref{change of variable for k}), we have
        \begin{align*}
             k(r_1,r_2,r_3)g^\prime n(x_2^\prime)a(t_2^\prime) \in NA \exp(r_1X_1+r_2X_2+r_3X_3) u_2^\prime,
        \end{align*}
        where $u_2^\prime\in K_0$ is defined in (\ref{u1 u2}), i.e.,
        \begin{align*}
            k^\prime g^\prime n(x_2^\prime)a(t_2^\prime) \in NA u_2^\prime.
        \end{align*}
        Following from \eqref{definition of ujprime}, we have
        \begin{align*}
            u_2^\prime = \begin{pmatrix}e^{i\theta_2}&0\\0&e^{-i\theta_2}\end{pmatrix}\begin{pmatrix}1/\sqrt{2}& \epsilon_2 i/\sqrt{2}\\ \epsilon_2i/\sqrt{2}&1/\sqrt{2}\end{pmatrix}.
        \end{align*}
        Following from (\ref{50'}), we have
        \begin{align*}
            k(r_1,r_2,r_3)g^\prime n\left( \eta_2(r_1,r_2,r_3) \right) a\left(\xi_2(r_1,r_2,r_3)\right) \in NA\begin{pmatrix}e^{i\tau_2( k(r_1,r_2,r_3),g^\prime)}&0\\0&e^{-i\tau_2( k(r_1,r_2,r_3),g^\prime)}\end{pmatrix}\begin{pmatrix}1/\sqrt{2}& \epsilon_2i/\sqrt{2}\\\epsilon_2i/\sqrt{2}&1/\sqrt{2}\end{pmatrix}.
        \end{align*}
        Here we use the fact $\sigma_2(k(r_1,r_2,r_3),g^\prime) = \epsilon_2$ because $\sigma_2(k(0,0,0),g^\prime)=\sigma_2(k^\prime,g^\prime) = \epsilon_2$.
         Therefore,
         \begin{align*}
            &\exp(r_1X_1+r_2X_2+r_3X_3)\begin{pmatrix}e^{i\theta_2}&0\\0&e^{-i\theta_2}\end{pmatrix}\begin{pmatrix}1/\sqrt{2}& \epsilon_2i/\sqrt{2}\\\epsilon_2i/\sqrt{2}&1/\sqrt{2}\end{pmatrix}n\left(e^{-t_2^\prime}(\eta_2(r_1,r_2,r_3)-x_2^\prime)\right)a\left(\xi_2(r_1,r_2,r_3)-t_2^\prime\right)\\
            =&\exp(r_1X_1+r_2X_2+r_3X_3) u_2^\prime n\left(e^{-t_2^\prime}(\eta_2(r_1,r_2,r_3)-x_2^\prime)\right)a(-t_2^\prime)a\left(\xi_2(r_1,r_2,r_3)\right)\\
            =&\exp(r_1X_1+r_2X_2+r_3X_3) u_2^\prime a(-t_2^\prime) n(-x_2^\prime)n\left(\eta_2(r_1,r_2,r_3)\right)a\left(\xi_2(r_1,r_2,r_3)\right)\\
            =& \exp(r_1X_1+r_2X_2+r_3X_3) u_2^\prime\cdot\left(k(r_1,r_2,r_3)g^\prime n(x_2^\prime) a(t_2^\prime)\right)^{-1}\cdot\left(k(r_1,r_2,r_3)g^\prime n\left(\eta_2(r_1,r_2,r_3)\right)a\left(\xi_2(r_1,r_2,r_3)\right) \right)\\
            \in&  NA\begin{pmatrix}e^{i\tau_2( k(r_1,r_2,r_3),g^\prime)}&0\\0&e^{-i\tau_2( k(r_1,r_2,r_3),g^\prime)}\end{pmatrix}\begin{pmatrix}1/\sqrt{2}& \epsilon_2i/\sqrt{2}\\\epsilon_2i/\sqrt{2}&1/\sqrt{2}\end{pmatrix}.
        \end{align*}
         The rest of the calculations for $\partial\eta_2/\partial r_i$ and $\partial\xi_2/\partial r_i$ are identical to those of $\partial\eta_1/\partial s_i$ and $\partial\xi_1/\partial s_i$ in Proposition \ref{derivatives of eta xi wrt si}. After applying Proposition \ref{jacobian u1 u2}, we obtain
        \begin{align*}
            B = \begin{pmatrix}
                   2e^{t_1^\prime} \sin(2\theta_1)&-2e^{t_1^\prime} \cos(2\theta_1)&0\\
                   -2\epsilon_1\cos(2\theta_1)&-2\epsilon_1\sin(2\theta_1)&0\\
                    2e^{h+t_2^\prime} \sin(2\theta_2)&-2e^{h+t_2^\prime} \cos(2\theta_2)&0\\
                   -2\epsilon_2e^h\cos(2\theta_2)&-2\epsilon_2e^h\sin(2\theta_2)&0\\
                   1&0&0\\
                   0&1&0\\
                   0&0&1
            \end{pmatrix}.
        \end{align*}
    \end{proof}
    
    \begin{rmk}
        We separate two calculations of $\partial \eta_1/\partial s_i$ and $\partial \xi_1/\partial s_i$, and $\partial \eta_2/\partial s_i$ and $\partial \xi_2/\partial s_i$ in different places (Proposition \ref{derivatives of eta xi wrt si} and Corollary \ref{Hessian of psi}). This is because $\partial \eta_1/\partial s_i$ and $\partial \xi_1/\partial s_i$ can be calculated without assuming $(k^\prime,g^\prime)$ to be critical, but the calculations for $\partial \eta_2/\partial s_i$ and $\partial \xi_2/\partial s_i$ use Proposition \ref{jacobian u1 u2}, which is based on the critical assumption.
    \end{rmk}
    
    \begin{rmk}
        The result in Corollary \ref{Hessian of psi} helps motivate Proposition \ref{eliminating degeneracy}, but it is not needed in our proof of Proposition \ref{bound for J}.
    \end{rmk}
    

\subsection{Degenerate estimate}\label{sec degenerate int est}
    We define
    \begin{align}\label{eq: sl2 perp}
        \fsl(2,\BR)^\perp = \left\{ \begin{pmatrix}
               iY_1&iY_2\\iY_3&-iY_1
        \end{pmatrix}\in  \fsl(2,\BC) : Y_i\in\BR\right\},
    \end{align}
    which is the subspace perpendicular to the Lie algebra of $H^\prime$ in $\fsl(2,\BC)$.
    Corollary \ref{Hessian of psi} shows that the degeneracy of the critical points of $\psi$ happens when $g$ is near $H^\prime$, which is the same as $\phi$.
    We also note that $\psi(k,g) = \psi(k,gg_0)$ for $g_0\in H^\prime$.
    Hence, we may study the degeneracy behavior of $\psi$ near $H^\prime$ by taking directional derivatives of $\psi(k,\exp(tY))$ for $Y\in \fsl(2,\BR)^\perp$ at $t=0$. We define
    \begin{align*}
        \cP_0 = \{ k\in K_0\, | \, (k,e) \notin\cS\}.
    \end{align*}
    Notice that if $k\in\cP_0$ and $Y\in\fsl(2,\BC)$, then $(k,\exp(tY)) \notin\cS$ for $t\in\BR$ sufficiently close to 0.
    
    \begin{proposition}\label{eliminating degeneracy}
        Let $k\in \cP_0$.
        If $Y=\begin{pmatrix}
               iY_1&iY_2\\iY_3&-iY_1
        \end{pmatrix}\in  \fsl(2,\BR)^\perp$ with $Y_1,Y_2,Y_3\in\BR$, then
        \begin{align*}
            \left.\frac{\partial}{\partial t}\psi(k,\exp(tY))\right|_{t=0} = \sigma_1(k) \left(2\eta_1(k)e^{-\xi_1(k)}Y_1 +e^{-\xi_1(k)}Y_2 - (e^{\xi_1(k)}+\eta_1(k)^2e^{-\xi_1(k)})Y_3\right).
        \end{align*}
        Here $\sigma_1(k)\in\{\pm1\}$ is defined in (\ref{50}).
        
        If we fix a compact subset $B\subset\cP_0$, then for any $k\in B$, $\partial\psi/\partial t(k,\exp(tY))|_{t=0}$ has a Hessian in the directions transverse to the critical set (which is $\operatorname{U}(1)$-invariant if nonempty) with determinant $\asymp \| Y \|^2$, where the implied constants depending on the set $B$. Hence, $\partial\psi/\partial t(k,\exp(tY))|_{t=0}$ is Morse-Bott 
        unless $Y=0$ in this case.
    \end{proposition}
    
    \begin{proof}
            We have
            \begin{align*}
                \left.\frac{\partial}{\partial t}\psi(k,\exp(tY))\right|_{t=0} =&\left.\frac{\partial}{\partial t}\phi\left(\eta_1(k),\xi_1(k),\eta_2(k,\exp(tY)),\xi_2(k,\exp(tY)),k,\exp(tY)\right)\right|_{t=0}\\
                =&\left.\frac{\partial}{\partial x_2}\phi\left(\eta_1(k),\xi_1(k),\eta_2(k,e),\xi_2(k,e),k,e\right)\cdot\frac{\partial}{\partial t}\eta_2(k,\exp(tY))\right|_{t=0}\\
                &+\left.\frac{\partial}{\partial t_2}\phi\left(\eta_1(k),\xi_1(k),\eta_2(k,e),\xi_2(k,e),k,e\right)\cdot\frac{\partial}{\partial t}\xi_2(k,\exp(tY))\right|_{t=0}\\
                &+\left.\frac{\partial}{\partial t}\phi\left(\eta_1(k),\xi_1(k),\eta_2(k,e),\xi_2(k,e),k,\exp(tY)\right)\right|_{t=0}.
            \end{align*}
            The first two terms vanish by (\ref{51'}), and $\eta_2(k,e)=\eta_1(k)$ and $\xi_2(k,e)=\xi_1(k)$, so we  are left with
            \begin{align*}
                \left.\frac{\partial}{\partial t}\psi(k,\exp(tY))\right|_{t=0} =&\left.\frac{\partial}{\partial t}\phi\left(\eta_1(k),\xi_1(k),\eta_1(k),\xi_1(k),k,\exp(tY)\right)\right|_{t=0}\\
                =& \left.\frac{\partial}{\partial t}A\left(k\exp(tY)n(\eta_1(k))a(\xi_1(k))\right)\right|_{t=0}.
            \end{align*}
            Write the first order approximation to the Iwasawa decomposition of $k\exp(tY)n(\eta_1(k))a(\xi_1(k))$ as
            \begin{align*}
                k\exp(tY)n(\eta_1(k))a(\xi_1(k)) = n\exp\left(tX_N+O(t^2)\right) a\exp\left(tX_A+O(t^2)\right) u\exp\left(tX_K+O(t^2)\right) ,
            \end{align*}
            where $n\in N$, $a\in A$, $u\in K_0$, $X_N\in\fn$, $X_A\in\fa$, and $X_K\in\fk$. As in (\ref{50}), we have
        \begin{align*}
               u= \begin{pmatrix}e^{i\tau_1(k)}&0\\0&e^{-i\tau_1(k)}\end{pmatrix}\begin{pmatrix}1/\sqrt{2}& \sigma_1(k) i/\sqrt{2}\\\sigma_1(k) i/\sqrt{2}&1/\sqrt{2}\end{pmatrix},
            \end{align*}
        where $\tau_1(k)\in\BR/2\pi\BZ$ and $\sigma_1(k)\in\{\pm 1\}$. 
            Rearranging and equating first-order terms gives
            \begin{align*}
               & Y = \Ad (n(\eta_1(k))a(\xi_1(k))u^{-1}a^{-1})X_N+ \Ad (n(\eta_1(k))a(\xi_1(k))u^{-1})X_A + \Ad (n(\eta_1(k))a(\xi_1(k)))X_K,\\
               &\Ad(ua(-\xi_1(k))n(-\eta_1(k))) Y=\Ad (a^{-1})X_N+ X_A + \Ad (u)X_K.
            \end{align*}
            As the right-hand side is the Iwasawa decomposition of the Lie algebra $\fsl(2,\BC) =\fn+\fa+\fk $, we see that $X_A$ is the projection of $\Ad(ua(-\xi_1(k))n(-\eta_1(k))) Y$ to $\fa$ under the Iwasawa decomposition. A calculation shows that
            \begin{align*}
                X_A = \sigma_1(k)\left(2\eta_1(k)e^{-\xi_1(k)}Y_1 + e^{-\xi_1(k)}Y_2 - (e^{\xi_1(k)}+\eta_1(k)^2e^{-\xi_1(k)})Y_3 \right)\begin{pmatrix}
                       1/2&0\\0&-1/2
                \end{pmatrix},
            \end{align*}
            so that
            \begin{align*}
                \left.\frac{\partial}{\partial t}A\left(k\exp(tY)n(\eta_1(k))a(\xi_1(k))\right)\right|_{t=0}= \sigma_1(k)\left(2\eta_1(k)e^{-\xi_1(k)}Y_1 + e^{-\xi_1(k)}Y_2 - (e^{\xi_1(k)}+\eta_1(k)^2e^{-\xi_1(k)})Y_3 \right).
            \end{align*}
            
            We now prove that $\partial\psi/\partial t(k,\exp(tY))|_{t=0}$ is Morse-Bott if $Y \neq 0$ under the assumption that $k$ lies in a compact set $B\subset\cP_0$. 
            By Proposition \ref{derivatives of eta xi wrt si}, $k\mapsto (\eta_1(k),\xi_1(k))$ from $\mathrm{U}(1)\backslash K_0$ to $\BR^2$ is a local diffeomorphism.
            Hence, the problem reduces to showing that
            \begin{align*}
                f(x,y) = 2xe^{-y}Y_1 + e^{-y}Y_2 - (e^{y}+x^2e^{-y})Y_3\quad\quad \text{ for }(x,y)\in (\eta_1,\xi_1) (B)
            \end{align*}
            only has nondegenerate critical points if $Y\neq 0$. Suppose there is a critical point $(x^\prime,y^\prime)$ of $f$, i.e.,
            \begin{align*}
                \frac{\partial f}{\partial x} (x^\prime,y^\prime) &= \frac{\partial f}{\partial y} (x^\prime,y^\prime) =  0,\\
                2e^{-y^\prime}Y_1 - 2x^\prime e^{-y^\prime} Y_3 &= -2x^\prime e^{-y^\prime}Y_1 - e^{-y^\prime}Y_2 - (e^{y^\prime}-x^{\prime 2}e^{-y^\prime})Y_3 = 0,
            \end{align*}
            so $Y_1 = x^\prime Y_3$ and $ Y_2 =- (e^{2y^\prime}+x^{\prime 2})Y_3 $. We may assume $Y_3 \neq 0$ because otherwise $Y_3=0$ would imply both $Y_1$ and $Y_2$ are $0$ and then $Y=0$. Moreover, for $x^\prime,y^\prime\in  (\eta_1,\xi_1) (B)$,
            \begin{align*}
                \| Y \|^2 = Y_1^2 + Y_2^2 +Y_3^2 = \left( 1 + x^{\prime2} +(e^{2y^\prime}+x^{\prime2})^2 \right) Y_3^2 \asymp_B Y_3^2.
            \end{align*}
            Then
            \begin{align*}
            \det\begin{pmatrix}
                        {\partial^2 f}/{\partial x^2} & {\partial^2f}/{\partial x \partial y}\\{\partial^2 f}/{\partial x \partial y}& {\partial^2 f}/{\partial y^2}
                \end{pmatrix} (x^\prime,y^\prime)  = 4 Y_3^2 \asymp \| Y \|^2.
           \end{align*}
    \end{proof}

    \begin{lemma}\label{lemma of estimate in Lie algebra}
        Given a compact set $B\subset \cP\backslash\cS$,
        there is an open neighborhood $0\in U \subset \fsl(2,\BR)^\perp$, such that for all $b\in C_c^\infty(\cP\backslash\cS)$ with $\supp(b)\subset B$, $Y\in U$, we have
        \begin{align*}
            \int_{K_0} b(k,\exp(Y)) \exp\left({is{\psi}(k,\exp(Y))}\right)dk\ll(1+s\|Y\|)^{-1},
        \end{align*}
         where the implied constant depends on $B$, the $L^\infty$-norms of $b$ and  finitely many of  its derivatives. 
    \end{lemma}
    \begin{proof}
        Define polar coordinates $X:\BR \times \BR/2\pi\BZ \times  \BR/2\pi\BZ \to \fsl(2,\BR)^\perp$ by
        \begin{align*}
            X(r,\gamma,\delta) = \begin{pmatrix}
               ir\cos(\gamma)&ir\sin(\gamma)\cos(\delta)\\ir\sin(\gamma)\sin(\delta)&-ir\cos(\gamma) 
        \end{pmatrix} = r\begin{pmatrix}
               i\cos(\gamma)&i\sin(\gamma)\cos(\delta)\\i\sin(\gamma)\sin(\delta)&-i\cos(\gamma) 
        \end{pmatrix}.
        \end{align*}
        We define
        \begin{align*}
            &\tilde{\cP}=K_0\times\BR \times \BR/2\pi\BZ \times  \BR/2\pi\BZ, \\
            &\tilde{\cS}=\left\{ (k,r,\gamma,\delta)\in \tilde{\cP}  : (k,\exp(X(r,\gamma,\delta))) \in \cS\right\},\\
            &\tilde{B} = \left\{ (k,r,\gamma,\delta)\in \tilde{\cP}  : (k,\exp(X(r,\gamma,\delta)))\in B\right\} \subset \tilde{\cP}\backslash \tilde{\cS}.
        \end{align*}
        The set $\tilde{B}$ is still compact.
        We define $\tilde{b}(k,r,\gamma,\delta)\in C^\infty_c(\tilde{\cP}\backslash\tilde{\cS})$ and $\tilde{\psi}(k,r,\gamma,\delta)\in C^\infty(\tilde{\cP}\backslash\tilde{\cS})$ to be the pullbacks of $b$ and $\psi$ under $X$, and then $\supp(b)\subset \tilde{B}$. We know that $\psi(k,g)$ vanishes when $g\in H^\prime$ so $\tilde{\psi}$ vanishes when $r=0$. Since $\tilde{\psi}$ is smooth (in fact, real analytic), we have that $\tilde{\psi}/r$ extends to a smooth function on $\tilde{\cP} \backslash \tilde{\cS}$.
        Proposition \ref{eliminating degeneracy} implies that $\tilde{\psi}/r$ is a Morse-Bott function with respect to the variable $k$ as $k$ lies in a compact set in $\cP_0$ when $r=0$.
        Hence, there exists some $\varepsilon>0$ such that $\tilde{\psi}/r$ is also Morse-Bott on the set $\tilde{B} \cap (K_0\times (-\varepsilon,\varepsilon) \times \BR/2\pi\BZ \times  \BR/2\pi\BZ)$.
        We define $U = X((-\varepsilon,\varepsilon)\times \BR/2\pi\BZ \times  \BR/2\pi\BZ)$.
        If $Y=X(r,\gamma,\delta)\in U\subset\fsl(2,\BR)^\perp$, then
        \begin{align}\label{eq: int over K, MB}
             \int_{K_0} b(k,\exp(Y)) \exp\left({is{\psi}(k,\exp(Y))}\right)dk=\int_{K_0}\tilde{b}(k,r,\gamma,\delta)\exp\left(  isr \left(\tilde{\psi}(k,r,\gamma,\delta)/r\right)\right) dk
        \end{align}
        with $\tilde{\psi}/r$ Morse-Bott. By applying Propositions \ref{eliminating degeneracy} and \ref{stationary phase}, if $|sr|\geq 1$, the right-hand side of \eqref{eq: int over K, MB} is $\ll |sr|^{-1}\asymp (s\|Y\|)^{-1}$. The result now follows by combining the above bound and the trivial bound.
    \end{proof}
    
    \begin{proposition}\label{estimate of oscillatory}
        Given a smooth compactly supported function $b\in C_c^\infty(\cP\backslash\cS)$ and $g_0\in H^\prime$, there is an open neighborhood $g_0\in U \subset \SL(2,\BC)$ such that for all  $g\in U$, we have
        \begin{align*}
            \int_{K_0} b(k,g)\exp\left({is{\psi}(k,g)}\right)dk\ll \left(1+s\,d(g,H^\prime)\right)^{-1}.
        \end{align*}
         The implied constant depends only on $g_0$,  the support of $b$,  and the $L^\infty$-norms of $b$ and finitely many of its derivatives.
    \end{proposition}
    
    \begin{proof}
        Let $U_0\subset \fsl(2,\BR)$ be an open bounded neighborhood of $0$ and let $B \subset \cP\backslash\cS$ be a compact set so that the supports of $b(\cdot,\cdot \exp(X_0) g_0) \in C_c^\infty(\cP\backslash\cS)$ are contained in $B$ for all $X_0\in U_0$. 
        We apply Lemma \ref{lemma of estimate in Lie algebra} to the set $B$ to get an open neighborhood $0\in U_Y\subset \fsl(2,\BR)^\perp$,
        and let $U=\exp(U_Y)\exp(U_0)g_0$. We assume $U_Y$ is small enough so that if $g=\exp(Y)\exp(X_0)g_0$ for $Y\in U_Y$ and $X_0 \in U_0$, we have $d(g,H^\prime)\asymp \|Y\|$, where the implied constant depends only on $g_0$. As ${\psi}(k,g) = \psi(k,\exp(Y)\exp(X_0)g_0) = {\psi}(k,\exp(Y))$, the result follows from Lemma \ref{lemma of estimate in Lie algebra}.
    \end{proof}

    \begin{corollary}\label{corollary estimate when g is close to H}
        Given a smooth compactly supported function $b\in C_c^\infty(\cP\backslash\cS)$, there is a constant $\delta_0 > 0$ such that for all  $g\in G_0$ satisfying $d(g,H^\prime) \leq \delta_0$ and $d(g,e)\leq 1$, we have
        \begin{align}\label{inequality of int K0}
            \int_{K_0} b(k,g)\exp\left({is{\psi}(k,g)}\right)dk\ll \left(1+s\,d(g,H^\prime)\right)^{-1}.
        \end{align}
         The implied constant depends only on the support of $b$,  and the $L^\infty$-norms of $b$ and finitely many of its derivatives.
    \end{corollary}
    
    \begin{proof}
        The subgroup $H^\prime$ is closed in $G_0$, so the set $B^\prime (2) = \{ g_0\in H^\prime \,|\, d(g_0,e)\leq 2 \}$
        is compact. There exist finitely many $g_0^{(i)}\in H^\prime$ with neighborhoods $U^{(i)} \subset \SL(2,\BC)$ given in Proposition \ref{estimate of oscillatory} so that $B^\prime (2)$ is covered by the union of $U^{(i)}$'s. Hence, (\ref{inequality of int K0}) holds for any $g\in \bigcup_i U^{(i)}$. Let $\delta_0>0$ be sufficiently small so that if $d(g,B^\prime(2)) \leq \delta_0$, then $g\in \bigcup_i U^{(i)} $ and so $g$ satisfies (\ref{inequality of int K0}). Now if $d(g,H^\prime) \leq \delta_0$ and $d(g,e)\leq 1$, then $d(g,B^\prime(2)) \leq \delta_0$ provided $\delta_0<1$.
    \end{proof}

\subsection{Bounds for $J(s,g)$}\label{sec bound for J}

   Now let us go back to estimate $J(s,g)$.  
    We will still let $\psi$ be as in \eqref{defn of psi}.

    \begin{lemma}\label{integral product lemma}
            Suppose that $c\in C_c^\infty(\BR^4\times \cP\backslash\cS)$ is a smooth and compactly supported function.
            Then there is another smooth and compactly supported function $c_1\in C_c^\infty(\cP\backslash \cS)$ such that \begin{align*}
                &\int_{K_0}\left(\int_{\BR^4} c(x_1,t_1,x_2,t_2,k,g) \exp\left(is{\phi}(x_1,t_1,x_2,t_2,k,g)\right) dx_1 dt_1 dx_2  dt_2 \right) dk\\
              = &s^{-2}\int_{K_0} c_1(k,g)\exp\left(is{\psi}(k,g)\right) dk + O(s^{-3}),
            \end{align*}
            where the implied constant depends on the support of $c$,  the $L^\infty$-norms of $c$ and finitely many of its derivatives.
        \end{lemma}
        \begin{proof}
            Since $(k,g)\in\cP\backslash\cS$, from (\ref{51}) and (\ref{51'}), $(\eta_1(k),\xi(k),\eta_2(k,g),\xi_2(k,g)) \in \BR^4$ is the unique critical point of $\phi(\cdot,k,g)$.  Moreover, it may be shown in the same way as the proof of Proposition \ref{Hessian} that the Hessian at this critical point is
            \begin{align*}
                D=\begin{pmatrix}
                       e^{-2\xi_1(k)} &&&\\
                       &1&&\\
                       &&-e^{-2\xi_2(k,g)}&\\
                       &&&-1
                \end{pmatrix}
            \end{align*}
            so that the critical point is uniformly nondegenerate for $(\eta_1(k),\xi(k),\eta_2(k,g),\xi_2(k,g),k,g)$ in the compact set $ \supp(c)$.
            We define $c_1\in C_c^\infty(\cP\backslash \cS)$ by
            \begin{align*}
                &c_1(k,g) = (2\pi)^{2}\exp\left({\xi_1(k)+\xi_2(k ,g)}\right) \cdot c(\eta_1(k),\xi_1(k),\eta_2(k,g),\xi_2(k ,g),k,g).
            \end{align*}
            Then the method of stationary phase (Proposition \ref{stationary phase}) gives
            \begin{align*}
                &\int_{\BR^4} c(x_1,t_1,x_2,t_2,k,g) \exp\left(is {\phi}(x_1,t_1,x_2,t_2,k,g)\right)  dx_1 dt_1 dx_2  dt_2\\
                =& s^{-2}c_1(k,g)\exp\left(is {\psi}(k,g)\right) + O(s^{-3}).
            \end{align*}
            Taking integral on $K_0$ proves the lemma.
        \end{proof}
    
    \begin{proof}[Proof of Proposition \ref{bound for J}]
         Let $\delta>0$ be a small constant to be chosen later.  If $d(g,H^\prime) \geq \delta$, then the analysis of critical points and Hessians of the phase function $\phi$ in Corollary \ref{corollary of critical point description} and Proposition \ref{Hessian} shows that
         \begin{align}\label{bound of J when dist is large}
             J(s,g)\ll_\delta s^{-3}
         \end{align}
         by applying Proposition \ref{stationary phase}.

         We consider the degenerate case and assume $d(g,H^\prime)<\delta$.  Recall that
         \begin{align*}
            J(s,g) = \int_{\BR^4}\int_{K_0} b(x_1,t_1,x_2,t_2,k)\exp(is\phi(x_1,t_1,x_2,t_2,k,g))  dk   dx_1 dt_1 dx_2 dt_2,
        \end{align*}
        where $b$ is a fixed smooth and compactly supported function on $\BR^4\times K_0$. By a partition of unity of $K_0\times G_0$, we have
        \begin{align*}
            b(x_1,t_1,x_2,t_2,k) = c(x_1,t_1,x_2,t_2,k,g) + c^\prime(x_1,t_1,x_2,t_2,k,g)
        \end{align*}
        so that $c\in C_c^\infty(\BR^4\times\cP\backslash\cS)$ and the support of $c^\prime$ does not contain any critical points of $\phi$. Hence,
        \begin{align*}
            J(s,g) &= \int_{\BR^4}\int_{K_0} c(x_1,t_1,x_2,t_2,k)\exp(is\phi(x_1,t_1,x_2,t_2,k,g))  dk   dx_1 dt_1 dx_2 dt_2 + O_N(s^{-N})\\
            &=s^{-2}\int_{K_0} c_1(k,g)\exp\left(is{\psi}(k,g)\right) dk + O(s^{-3}),
        \end{align*}
        where $c_1\in C_c^\infty(\cP\backslash\cS)$ is obtained from Lemma \ref{integral product lemma}. If we let the cutoff function in Corollary \ref{corollary estimate when g is close to H} be $c_1$ and let $\delta$ be $\delta_0$ in the corollary, then
        \begin{align}\label{bound of J when dist is small}
             J(s,g) \ll s^{-2}\left(1+s\,d(g,H^\prime)\right)^{-1}.
        \end{align}
         By combining the bounds (\ref{bound of J when dist is large}) and (\ref{bound of J when dist is small}) with the choice $\delta=\delta_0$, we complete the proof.
    \end{proof}

    \bibliographystyle{plain} 
    \bibliography{refs.bib} 

    \end{document}